\documentclass{article}
\pdfoutput=1 
\usepackage[english]{babel}

\usepackage[letterpaper,top=2cm,bottom=2cm,left=3cm,right=3cm,marginparwidth=1.75cm]{geometry}

\usepackage{amsmath}
\usepackage{amsfonts}
\usepackage{amsthm}
\usepackage{authblk}
\usepackage{graphicx}
\usepackage{bm}
\usepackage{mathrsfs,mathtools}
\usepackage{xspace}
\usepackage[colorlinks=true, allcolors=blue]{hyperref}
\usepackage{cleveref}
\usepackage{subfigure}
\usepackage{comment}
\usepackage{algorithm}
\usepackage{algpseudocode}
\usepackage[table,xcdraw]{xcolor}

\usepackage{color}
\usepackage{circuitikz}
\usepackage{tikz}
\usetikzlibrary{calc,positioning,shapes}
\usetikzlibrary{patterns,decorations.pathmorphing,decorations.markings}
\usetikzlibrary{external}
\usepackage{pgfplots}
\pgfplotsset{compat=newest}
\usepackage[margin=10pt,font=small,labelfont=bf,labelsep=endash]{caption}
\usepackage{subcaption}

\usetikzlibrary{external}
\usetikzlibrary{calc,positioning,shapes}
\usetikzlibrary{patterns,decorations.pathmorphing,decorations.markings}
\usetikzlibrary{intersections, backgrounds}
\usetikzlibrary{circuits}
\usetikzlibrary{circuits.ee.IEC}
\usepackage{pgfplots}
\pgfplotsset{compat=newest}

\definecolor{color0}{rgb}{0.12156862745098,0.466666666666667,0.705882352941177}
\definecolor{color1}{rgb}{1,0.498039215686275,0.0549019607843137}
\definecolor{color2}{rgb}{0.172549019607843,0.627450980392157,0.172549019607843}
\definecolor{color3}{rgb}{0.83921568627451,0.152941176470588,0.156862745098039}
\definecolor{color4}{rgb}{0.580392156862745,0.403921568627451,0.741176470588235}
\definecolor{color5}{rgb}{0,0,0}

\definecolor{mycolor1}{rgb}{0.00000,0.44700,0.74100}
\definecolor{mycolor2}{rgb}{0.85000,0.32500,0.09800}
\definecolor{mycolor3}{rgb}{0.92900,0.69400,0.12500}
\definecolor{mycolor4}{rgb}{0.46600,0.67400,0.18800}
\definecolor{mycolor5}{rgb}{0.49400,0.18400,0.55600}

\newcommand{\lineWidth}{1.2pt}
\newcommand{\imageWidth}{2.0in}
\newcommand{\imageHeight}{1.8in}

\usepackage{booktabs}
\usepackage{paralist}
\usepackage{soul}
\usepackage{listings}

\newcommand{\cS}{\mathcal{S}}

\def\eps{\varepsilon}

\renewcommand{\Re}{{\mbox{\rm Re}}}
\renewcommand{\Im}{{\mbox{\rm Im}}}

\newcommand{\matlab}{{\sc Matlab}}
\newcommand{\imagunit}{{\bf i}}
\newcommand{\Id}{{\fI}}

\newcommand{\conj}[1]{\overline{#1}}

\newcommand{\ophi}{\phi}

\newcommand{\abbr}[1]{\textsf{#1}\xspace}

\newcommand{\QWF}{\abbr{QWF}}

\newcommand{\MOR}{\abbr{MOR}}

\newcommand{\DAE}{\abbr{DAE}}
\newcommand{\DAEs}{\abbr{DAEs}}
\newcommand{\ODE}{\abbr{ODE}}
\newcommand{\ODEs}{\abbr{ODEs}}
\newcommand{\CIM}{\abbr{CIM}}
\newcommand{\CIMs}{\abbr{CIMs}}
\newcommand{\PDE}{\abbr{PDE}}
\newcommand{\PDEs}{\abbr{PDEs}}
\newcommand{\IVP}{\abbr{IVP}}
\newcommand{\QR}{\abbr{QR}}

\theoremstyle{plain}
\newtheorem{theorem}{Theorem}[section]
\newtheorem{proposition}[theorem]{Proposition}
\newtheorem{lemma}[theorem]{Lemma}
\newtheorem{corollary}[theorem]{Corollary}
\newtheorem{remark}[theorem]{Remark}

\newtheorem{problem}[theorem]{Problem}

\newcommand{\zeroVec}{\mathbf{0}}
\newcommand{\zeroMat}{\mathbf{0}}
\newcommand{\gleMat}[1]{\mathbf{\mathscr{#1}}}
\newcommand{\lapOp}{\gleMat{L}}
\newcommand{\lapVar}{z}
\newcommand{\intvar}{s}
\newcommand{\dintvar}{\integrate{\intvar}}

\newcommand{\tol}{\abbr{tol}}
\newcommand{\indDAE}{\nu}
\newcommand{\difW}{\hat{\mathbf{V}}}
\newcommand{\impW}{\hat{\mathbf{W}}}
\newcommand{\diff}{\mathrm{diff}}
\newcommand{\imp}{\mathrm{imp}}

\newcommand{\N}{\ensuremath\mathbb{N}}

\newcommand{\R}{\ensuremath\mathbb{R}}
\newcommand{\C}{\ensuremath\mathbb{C}}

\newcommand{\smoothFunctions}[3][]{\ifthenelse{\equal{#1}{}}{\mathcal{C}^{#2}}{\mathcal{C}_{#1}^{#2}}(#3)}

\newcommand{\dist}[1][]{\ifthenelse{\equal{#1}{}}{\mathbb{D}}{#1_{\mathbb{D}}}}

\newcommand{\Frob}{\mathrm{F}}

\newcommand{\T}{\ensuremath\mathsf{T}}

\newcommand{\integrate}[1]{\mathrm{d}#1}

\newcommand{\dz}{\integrate{z}}

\newcommand{\domega}{\integrate{\omega}}

\newcommand{\dtau}{\integrate{\tau}}

\DeclareMathOperator{\ee}{e}

\DeclareMathOperator{\image}{img}
\newcommand{\img}{\image}

\newcommand{\calC}{\mathcal{C}}

\newcommand{\calN}{\mathcal{N}}
\newcommand{\calO}{\mathcal{O}}
\newcommand{\calP}{\mathcal{P}}

\newcommand{\calV}{\mathcal{V}}
\newcommand{\calW}{\mathcal{W}}

\newcommand{\stx}{\ensuremath{\bm{x}}}

\newcommand{\stateDim}{n}

\newcommand{\stateDimRed}{r}
\newcommand{\inp}{\ensuremath{\bm{u}}}
\newcommand{\inpDim}{m}

\newcommand{\outDim}{p}
\newcommand{\source}{\ensuremath{\bm{f}}}
\newcommand{\inSol}{\stx_0}

\newcommand{\fe}{\ensuremath{\bm{e}}}

\newcommand{\fp}{\ensuremath{\bm{p}}}

\newcommand{\fu}{\ensuremath{\bm{u}}}
\newcommand{\fv}{\ensuremath{\bm{v}}}

\newcommand{\fx}{\ensuremath{\bm{x}}}
\newcommand{\fy}{\ensuremath{\bm{y}}}

\newcommand{\fA}{\ensuremath{\bm{A}}}
\newcommand{\fB}{\ensuremath{\bm{B}}}
\newcommand{\fC}{\ensuremath{\bm{C}}}
\newcommand{\fD}{\ensuremath{\bm{D}}}
\newcommand{\fE}{\ensuremath{\bm{E}}}

\newcommand{\fG}{\ensuremath{\bm{G}}}

\newcommand{\fI}{\ensuremath{\bm{I}}}
\newcommand{\fJ}{\ensuremath{\bm{J}}}
\newcommand{\fK}{\ensuremath{\bm{K}}}
\newcommand{\fL}{\ensuremath{\bm{L}}}
\newcommand{\fM}{\ensuremath{\bm{M}}}
\newcommand{\fN}{\ensuremath{\bm{N}}}

\newcommand{\fP}{\ensuremath{\bm{P}}}

\newcommand{\fS}{\ensuremath{\bm{S}}}
\newcommand{\fT}{\ensuremath{\bm{T}}}

\newcommand{\fV}{\ensuremath{\bm{V}}}
\newcommand{\fW}{\ensuremath{\bm{W}}}
\newcommand{\fX}{\ensuremath{\bm{X}}}
\newcommand{\fY}{\ensuremath{\bm{Y}}}

\newcommand{\fdelta}{\ensuremath{\bm{\delta}}}

\newcommand{\flambda}{\ensuremath{\bm{\lambda}}}

\newcommand{\frho}{\ensuremath{\bm{\rho}}}

\newcommand{\fDelta}{\ensuremath{\bm{\Delta}}}

\newcommand{\fPi}{\ensuremath{\bm{\Pi}}}

\newcommand{\prmtr}{\boldsymbol{\mu}}
\newcommand{\prmtrSet}{\calP}
\newcommand{\prmtrDim}{d}

\newcommand{\Col}{\mathrm{Col}}
\newcommand{\lutc}{\underset{\sim}{<}}

\title{Contour integral methods and structured perturbations for linear differential-algebraic equations}

\author[1]{Nicola Guglielmi}
\author[2]{Mattia Manucci}

\affil[1]{\small {Division of mathematics}, {Gran Sasso Science Institute}, {{Viale Francesco Crispi~7}, {67100}, {L'Aquila}, {Italy}}}
\affil[2]{{Institute for Applied and Numerical Mathematics (IANM)}, {Karlsruhe Institute of Technology}, {Englerstrsse~2}, {76131}, {Karlsruhe}, {Germany}}

\begin{document}
\maketitle

\begin{abstract}
We generalize the contour integral methods (CIM) framework to the time integration of linear dynamical systems that are subject to algebraic constraints at all times during their evolution. The proposed approach relies on applying the Laplace transform to the Cauchy problem associated with a linear system of differential–algebraic equations (DAE), and subsequently reconstructing the time-domain solution by approximating the inverse Laplace transform via a suitable quadrature rule. This procedure yields an efficient and accurate alternative to classical Runge–Kutta schemes, which are well known to exhibit order reduction in accuracy when applied to DAE.

In the second part of the paper, we address linear parametric DAE and propose an efficient strategy for tuning the integration contour in the CIM framework using suitable structured-unstructured pseudospectral computations. This allows the identification of a single integration profile capable of approximating an entire family of parametric solutions, thereby facilitating the efficient application of model order reduction techniques.

Finally, numerical experiments are presented to validate the proposed methodology and support the theoretical findings.
\end{abstract}

\section{Introduction}

In this work, we consider the approximation of the solution 
$\stx(t)\in\R^{\stateDim}$ of a linear dynamical system subject to algebraic constraints. 
The system has the form
\begin{equation}\label{eqn:DAE:0}
   \fE \dot{\stx}(t)  = \fA\stx(t) +\source(t), 
\end{equation}
where $\fE,\fA\in\R^{\stateDim\times\stateDim}$ and 
$\source:\R\rightarrow \R^{\stateDim}$. Because there are algebraic constraints that must hold for every $t \in \R$, it follows directly that the matrix $\fE$ is singular.
Systems of the form \eqref{eqn:DAE:0} are commonly referred to as 
\emph{differential--algebraic equations} (\DAE).

As a representative example, we consider the semi-discrete Stokes equations, for which the incompressibility (mass-conservation) condition imposes algebraic constraints on the evolution of the system. More broadly, \DAE systems arise in a wide range of applications, including robotic manipulators, traffic-flow modeling, automatic gear-shifting mechanisms, and electrical power systems; see, for example, \cite{KunM24} and the references therein.

Suppose that the objective is to approximate $\stx(t)$ solely at a prescribed time instant $t = T$ or within a specified time interval $t \in [T, \Lambda T]$ with $\Lambda > 1$. In this setting, conventional time-stepping integrators and space–time variational formulations can carry out a large number of evaluations at intermediate time levels that are not of direct interest. In addition, time-integration schemes must address structural challenges that are inherent to \DAE systems, such as order reduction and stability limitations arising from the treatment of constrained variables; see \cite[Sec.~4.2]{BreCP95} and \cite{HaiLR89,HaiW96}.
Motivated by these considerations, we investigate the use of 
\emph{contour integral methods} (\CIM) for approximating $\stx(t)$. 
These methods compute the solution by numerically evaluating its inverse Laplace transform, 
i.e., by applying suitable quadrature rules to the contour integral representation given by the inverse Laplace formula. 
In particular, we rely on the approach proposed in \cite{GugLM21}, 
where the integration contour is determined using a pseudospectral roaming technique based on selected weighted pseudospectral level sets of the leading operator.

A main benefit of \CIM is that they enable one to approximate the solution directly at selected time instants, or over appropriate time windows, without the need to perform intermediate time integration.
They have been successfully applied to diffusion equations 
\cite{GavM01,LopP04,LopPS06,SheST03}, 
fractional-in-time problems \cite{Col22,ColA22}, 
and convection-diffusion equations \cite{GugLN18,GugLM21}. 
These problems are associated with sectorial operators, 
for which \CIM are particularly effective due to contour deformation techniques. 

A fundamental assumption to ensure the efficient evaluation of the approximate solution via \CIM is that the Laplace transform ($\lapOp$) of the solution of \eqref{eqn:DAE:0} exhibits decaying asymptotic behavior. Denoting by $\hat{\stx}\vcentcolon=\lapOp(\stx)$ and $\hat{\source}\vcentcolon=\lapOp(\source)$ the Laplace transforms of $\stx$ and $\source$, respectively, and using \eqref{eqn:DAE:0}, the quantity $\hat{\stx}$ is given by
\begin{equation}\label{eqn:lap}
    \hat{\stx}(\lapVar)
    =
    \left(\lapVar\fE-\fA\right)^{-1}
    \left(\fE\stx_0+\hat{\source}(\lapVar)\right).
\end{equation}
To analyze the behavior of $\hat{\stx}$ as $|\lapVar|\rightarrow\infty$, it is essential to investigate the corresponding generalized resolvent 
$\|(\lapVar\fE-\fA)^{-1}\|$. 
Here, $\|\cdot\|$ denotes the induced $2$-norm. 
If $\fE$ is invertible, the resolvent exhibits a decay of order $\calO(|\lapVar|^{-1})$. 
In contrast, for a general \DAE, the resolvent may fail to decay and can even display unbounded growth.
Our first contribution is to establish that, under standard regularity assumptions on the matrix pencil $(\fA,\fE)$ and on the source term $\source$, 
\CIM provide an efficient and accurate approximation framework for systems of the form \eqref{eqn:DAE:0}, despite the potentially unfavorable asymptotic behavior of the generalized resolvent for \DAE.

\medskip

In the second part of the paper, we consider parametric linear systems, including \DAE. Specifically, 
\begin{equation}\label{eqn:DAE:par}
   \fE \dot{\stx}(t;\prmtr)
   =
   \fA(\prmtr)\stx(t;\prmtr) +\source(t;\prmtr), 
\end{equation}
where $\prmtr\in\prmtrSet\subset\R^{\prmtrDim}$, 
$\fA:\prmtrSet\rightarrow\R^{\stateDim\times\stateDim}$, 
and 
$\source:\R\times\prmtrSet\rightarrow \R^{\stateDim}$. Such systems arise, for example, in semi-discrete Stokes problems where the diffusion coefficient acts as a physical parameter. 
More generally, parametric time-dependent linear systems appear in structural mechanics, diffusion processes, electromagnetism, fluid dynamics, control, and many other applications. Note that one could also naturally allow $\fE:\prmtrSet\rightarrow\R^{\stateDim\times\stateDim}$. However, since in many problems the parameters appear only in $\fA$, and to keep the presentation of our results more streamlined, we mainly focus on the case where $\fE$ is non-parametric living the parametric case to specific remarks towards the manuscript.

In multi-query settings, such as optimal control, shape optimization, or uncertainty quantification, these systems must be evaluated for many parameter values. 
When the dimension $\stateDim$ is large, repeated evaluations become computationally prohibitive. 
Model order reduction (\MOR) techniques aim to address this challenge by constructing reduced systems whose evaluation cost is independent of $\stateDim$; see \cite{BenGW15,HesRS16,GlaMU17}. In scenarios where the solution is required only at a specific time instant $t=T$ or within a time window $[T,\Lambda T]$, time-stepping integrators may incur unnecessary computational overhead. 
In \cite{GugM23}, \CIM were proposed as time integrators within projection-based reduced order models. 
A substantial speed-up can be achieved provided that a single (or a small number of) contour integration profile(s) can be used uniformly on the parameter set of interest. This motivates the following question.

\begin{problem}\label{pro1}
Given a parameter $\prmtr_0\in\R^{\prmtrDim}$ and its associated contour $\Gamma_{\prmtr_0}$, 
for which other parameters $\prmtr\in\R^{\prmtrDim}$ can the same contour $\Gamma_{\prmtr_0}$ be retained without compromising the approximation quality of the employed \CIM?
\end{problem}
The quality of the \CIM-based approximation of linear time-dependent systems depends, among other factors, in a critical manner on the behavior of the generalized resolvent, which can display highly irregular structures over the complex plane and throughout the parameter domain; see \cite{GugLN18}. To tackle Problem~\ref{pro1}, we employ a structured–unstructured perturbation framework for eigenvalue problems; see \cite{book}. This framework enables a precise characterization of the subset of the parameter space for which the generalized resolvent remains uniformly bounded above by a prescribed threshold $\varepsilon^{-1}$.

\subsection{Organization of the manuscript}
The article is structured as follows. In \Cref{sec:constr} we introduce the \CIM methodology and demonstrate its suitability for the \DAE setting. In \Cref{sec:str-uns} we develop a structured–unstructured eigenvalue perturbation framework to characterize the parameter set over which the generalized resolvent remains uniformly bounded, as this property is pivotal both for the construction of the conformal map in \CIM and for the convergence behavior of the employed quadrature rule. In \Cref{sec:num}, we evaluate the methodologies proposed in \Cref{sec:constr} and \Cref{sec:str-uns} by applying them to a range of numerical experiments. Finally, in \Cref{sec:conc} we summarize the main findings and present our concluding remarks.

\subsection{Notation}
Given a generic matrix $\fM$ we denote by $\Col(\fM)$ the column space of $\fM$. For $a,b\in\R$ the relation $a\lutc b$ indicates that $a$ is smaller than $b$ up to a constant that is independent of both $a$ and $b$. With $\fI_{\stateDim}$, we denote the square identity matrix of the appropriate size $\stateDim$. The preimage of $\fA$ with respect to a linear subspace $\calN \subseteq \R^{\stateDim}$ is denoted with
\begin{equation}\label{eqn:pre:img}
	\fM^{-1}(\calN)\coloneqq\{\fx\in\R^{\stateDim} \mid \fM\fx\in\calN\}.
\end{equation}
We denote by 
\begin{equation*}
\langle \fX,\fY \rangle =\sum_{i,j} \conj{\fx}_{ij}\fy_{ij} = {\rm tr}(\fX^* \fY),
\end{equation*} 
the inner product in $\C^{n, n}$ that induces the Frobenius norm $\| \fX \|_\Frob = \langle \fX,\fX \rangle^{1/2}$. 

\section{Time integration of \DAE via \CIM} \label{sec:constr}
The system \eqref{eqn:DAE:0} provided with an initial condition gives the \emph{initial value problem} (\IVP) of the form
\begin{equation}
	\label{eqn:DAE}
	\left\{\quad \begin{aligned}
		\fE \dot{\stx}(t)  &= \fA\stx(t) +\source(t),\\  \stx(t_0) &= \inSol,
	\end{aligned}\right.
\end{equation}
where the symbols $\stx(t)\in\R^{\stateDim}$, $\source(t)\in\R^{\stateDim}$, and $\inSol\in\R^{\stateDim}$ denote the \emph{solution} or \emph{state} of the system at time $t$, the \emph{source} or \emph{forcing} term evaluated at time $t$, and the \emph{initial solution}, respectively. We recall that $\fE$ is assumed to be singular. We assume that the finite eigenvalues $\lambda$ of the matrix pair $(\fE,\fA)$ have a negative real part, i.e., $\lambda\in\C_{-}$ for all $\lambda$ such that $\fA\fv=\lambda\fE\fv$ with $\fv\in\C^{\stateDim}$. In the \DAE contest one normally refers to the ``infinite eigenvalues" as the eigenvalues associated with the cases where $\fE\fv=\zeroVec$ but $\fA\fv\neq\zeroVec$ and $\fv\neq\zeroVec$.

In \Cref{subsec:CIM}, we provide a concise introduction to \CIM and emphasize the key elements that warrant examination for their applicability in the context of \DAE. In \Cref{subsec:DAE} we recall some standard tools for the study of \DAE. Finally, in \Cref{subsec:CIM-DAE} we analyze \CIM for \DAE and provide conditions for their efficient applicability.

\subsection{The \CIMs in a nutshell}\label{subsec:CIM}
We assume the existence of the Laplace transform of $\source$ and that it admits a bounded analytic extension to a suitable region of the complex plane outside the
finite eigenvalues of the matrix pair $(\fE,\fA)$. We then apply the Laplace transform operator $\lapOp$ to \eqref{eqn:DAE}, which gives, for the Laplace transform of $\stx$, the expression in \eqref{eqn:lap}. To recover the solution in the time domain, we apply the inverse Laplace transform, which gives
\begin{equation}\label{eqn:inv:Lap}
    \stx(t)\;=\;\frac{1}{2\pi\imagunit}\int_{\gamma-\imagunit\infty}^{\gamma-\imagunit\infty} \ee^{\lapVar t}\hat{\stx}(\lapVar)\;\domega,
\end{equation}
for certain $\gamma>0$. Then, assuming
\begin{enumerate}
    \item the singularity of the integrand function in \eqref{eqn:inv:Lap} lies in a sectorial region of the complex place with $\Re(\lapVar)\le \gamma$;\label{ass:1}
    \item the integrand function in \eqref{eqn:inv:Lap} decays as $|\lapVar|\rightarrow\infty$;\label{ass:2}
\end{enumerate}
following the idea first introduced in \cite{But57, Tal79}, we can deform the vertical line of integration in \eqref{eqn:inv:Lap} to the contour $\Gamma$, with $\Gamma$ an open piecewise smooth curve running from $-\imagunit\infty$ to $+\imagunit\infty$ surrounding all singularities of $\hat{\stx}$ in \eqref{eqn:lap}; thus we get
\begin{equation}\label{eqn:Bro:int}
        \stx(t)\;=\;\frac{1}{2\pi\imagunit}\int_{\Gamma} \ee^{\lapVar t}\hat{\stx}(\lapVar)\;\dz.
\end{equation}
The integral \eqref{eqn:Bro:int} is called the Bromwich integral and, to approximate it, we parameterize the
integration contour $\Gamma$ with a conformal map $\lapVar: \R\rightarrow \Gamma$ such that
\begin{equation}\label{eqn:int2}
    \int_{\Gamma}\ee^{\lapVar t}\hat{\stx}(\lapVar)\;\dz\;=\;\int_{\R}\fG(\intvar)\;\dintvar,\quad \fG(\intvar)\;\vcentcolon=\;\ee^{\lapVar(\intvar) t} \hat{\stx}(\lapVar(\intvar))\frac{\partial\lapVar}{\partial \intvar}(\intvar).
\end{equation}
As already mentioned, we are interested in approximating $\stx$ at a specific time $T$, thus we fix $t=T$ and assume that we have a target precision denoted as $\tol$ for the required approximation. This allows us to truncate the integral in \eqref{eqn:int2} and thus only to consider a portion of the Bromwich integral that we parameterize by $[-c\pi, c\pi]$. This is
\begin{equation*}
    \int_{\R} \fG(\intvar)\;\dintvar\;\approx\; \int_{-c\pi}^{c\pi} \fG(\intvar)\;\dintvar,
\end{equation*}
for a certain truncation parameter $c \in (0, c_{\max})$, which we determine by solving the non-linear equation $\|\fG(c\pi)\| = \tol$. Finally, the application of a quadrature formula to approximate \eqref{eqn:Bro:int} provides a numerical approximation of $\stx$, for a given time $T$, or even time windows of
the form $[T, \Lambda T]$, $\Lambda>1$, without the need to compute it at intermediate time instants. For example, an application of the trapezoidal rule provides the desired approximation $\stx_N(T)$ of $\stx(T)$, where $\stx_N(T)$ reads as
\begin{equation}\label{eqn:quad:app}
    \stx_N(T)\;\vcentcolon=\;\frac{c}{\imagunit N}\sum_{j=1}^{N-1}\fG(\intvar_j), \text{ with }\intvar_j\;=\;-c\pi+j\frac{2c\pi}{N},\quad j=1,\ldots,N-1.
\end{equation}
Note that the evaluation of each term in the summation \eqref{eqn:quad:app} involves solving the linear system corresponding to the matrix $\fA - \lapVar(\intvar_j)\fE$, as the quantity $\hat \stx(\lapVar(\intvar_j))$ is obtained by the relation \eqref{eqn:lap}. An advantage of the method we propose is that these computations can be easily parallelized since the $N-1$ systems are independent of each other. Furthermore, since the integrand is conjugate symmetric, the number of addends, and thus the number of linear
systems to be solved, can be halved. We also note that, despite the use of a simple trapezoidal quadrature rule, it can be shown (see \cite[Thm.~2]{GugLN18}) that the error in the spectral norm between $\stx_N(T)$ and $\stx(T)$ decays exponentially with respect to the number of quadrature points employed due to the analyticity of the integrand function (see \cite{TreW14}), leading to solving (possibly in parallel) only a few linear systems.

To apply \CIM to \DAE we have to state conditions under which assumptions \cref{ass:1}~and~\ref{ass:2} hold. For linear systems of ordinary differential equations (\ODE), the singularities of the integrand function can be characterized in terms of the eigenvalues of $\fA$ and the poles of $\hat{\source}$, while \cref{ass:2} is always verified under a specific limited growth condition of the Laplace transform of the source term. In the \DAE context, the validation of such hypotheses requires a thorough and rigorous examination. In fact, we still have to account for the singularities of $\hat{\source}$, and then we should look for the points for which the matrix pencil $\lapVar\fE-\fA$ is not invertible. Moreover, the asymptotic behavior of $\|(\lapVar\fE-\fA)^{-1}\|$ also has to be determined in order to have conditions under which \cref{ass:2} is satisfied. We elaborate on these aspects in \Cref{subsec:CIM-DAE}.

\subsection{Some standard tools for \DAE}\label{subsec:DAE}
First, to ensure the existence and uniqueness of the solution of \eqref{eqn:DAE}, the matrix pair $(\fE,\fA)$ must satisfy certain properties; see, for instance, \cite[Cha.~2]{KunM06}. In more detail, we assume that the matrix pair $(\fE,\fA)$ is \emph{regular}, i.e., $\det(\lapVar\fE-\fA)\in\C[\lapVar]\setminus\{0\}$ which means that the determinant of the matrix pencil is not the zero polynomial. In this case, one can show that in the space of piecewise-smooth distribution \cite{Tre09} the initial trajectory problem associated with the \DAE \eqref{eqn:DAE} has a unique solution for any initial value and any right-hand side. Regularity can be characterized by the Weierstrass form \cite{Gan59} or the slightly simplified \emph{quasi-Weierstrass form} (\QWF) \cite{BerIT12}.

\begin{theorem}[Quasi-Weierstrass Form, \cite{BerIT12}]\label{teo:QWF}
	A matrix pair $(\fE,\fA)\in\R^{\stateDim\times\stateDim}\times\R^{\stateDim\times\stateDim}$ is regular if and only if there exist invertible matrices $\fS,\fT\in\R^{\stateDim\times\stateDim}$ such that
	\begin{equation}
		\label{eqn:QWF}
		\left(\fS\fE\fT, \fS\fA\fT\right) = \Bigg(   \begin{bmatrix}
			\Id_{\stateDim_{\fJ}} &\zeroMat\\
			\zeroMat &\fN
		\end{bmatrix},\begin{bmatrix}
			\fJ&\zeroMat\\
			\zeroMat & \Id_{\stateDim_{\fN}}
		\end{bmatrix}  \Bigg),
	\end{equation}
	where $\fN\in\R^{\stateDim_{\fN} \times \stateDim_{\fN}}$ is nilpotent with nilpotency index $\indDAE$ and $\fJ\in\R^{\stateDim_{\fJ} \times \stateDim_{\fJ} }$, with $\stateDim_{\fJ}=\stateDim-\stateDim_{\fN}$. The decoupling \eqref{eqn:QWF} is called the quasi-Weierstrass form, and $\indDAE$ is also called the \DAE index.
\end{theorem}

\begin{remark}
	The assumption that the finite eigenvalues of the matrix pair $(\fE,\fA)$ have a negative real part implies that the matrix $\fJ$ appearing in \eqref{eqn:QWF} is asymptotically stable. 
\end{remark}
Under the regularity assumption, for every regular matrix pair $(\fE,\fA)$ there exist unique subspaces $\calV,\calW \subseteq \R^\stateDim$ with $\calV \oplus\calW = \R^n$ such that for any choice of full column rank matrices $\fV$, $\fW$ with $\Col(\fV) = \calV$ and $\Col(\fW) = \calW$, the nonsingular
matrices $\fT = [\fV,\, \fW]$ and $\fS = [\fE\fV,\, \fA\fW]^{-1}$ transform \eqref{eqn:DAE} into a decoupled \DAE according to \eqref{eqn:QWF} with an \ODE part, often denoted as slow subsystem, of the form
\begin{equation}\label{eqn:DAE:slow}
    \dot{\stx}^{\diff}(t) \;=\; \fJ\stx^{\diff}(t)+\source_{{\diff}}(t),\quad \source_{{\diff}}(t)\;\vcentcolon=\;\begin{bmatrix}
        \fI_{\stateDim_{\fJ}}&\zeroVec,
    \end{bmatrix}\fS\source(t),
\end{equation}
and a nilpotent \DAE part, which in contrast is the fast subsystem, of the form
\begin{equation}\label{eqn:DAE:fast}
   	\fN\dot{\stx}^{\imp}(t) \;=\; \stx^{\imp}(t)+\source_{{\imp}}(t),\quad
	\source_{{\imp}}(t)\;\vcentcolon=\;\begin{bmatrix}
        \zeroVec& \fI_{\stateDim_{\fN}}
    \end{bmatrix}\fS\source(t),
\end{equation}
with $\stx(t) = \fT(\stx^{\diff}(t) \oplus\stx^{\imp}(t))$ for all $t\in\R_{\ge 0}$. This can be used to derive an explicit solution formula; see \cite[Cha.~2]{KunM06}. For the differential part coming from \eqref{eqn:DAE:slow} we have
\begin{equation}\label{eqn:sol:sub:slow}
    \stx^{\diff}(t)\;=\;\ee^{\fJ t} \stx^{\diff}_0+\int_{0}^{t}\ee^{\fJ(t-\tau)}\source_{{\diff}}\;\dtau,\quad \stx^{\diff}_0\;=\;\begin{bmatrix}
        \fI_{\stateDim_{\fJ}}&\zeroVec
    \end{bmatrix}\fT^{-1}\stx_0;
\end{equation}
while, for \eqref{eqn:DAE:fast} we have 
\begin{equation}\label{eqn:sol:sub:fast}
    \stx^{\imp}(t)\;=\;-\sum_{j=0}^{\indDAE-1} \fN^j\source^{(j)}_{\stx^{\imp}}(t).
\end{equation}
where we require $\source_{{\imp}}$ to be $\indDAE-1$ piecewise continuously differentiable. In particular, the fast subsystem~\eqref{eqn:DAE:fast} generates an inconsistency in the initial value, indeed by \eqref{eqn:sol:sub:slow}, \eqref{eqn:sol:sub:fast}, and the fact that $\stx = \fT(\stx^{\diff} \oplus\stx^{\imp})$ we get
\begin{equation}\label{eqn:ini:sol}
    \stx(0^{-})\;=\;\fT\left(\begin{bmatrix}
        \fI_{\stateDim_{\fJ}}&\zeroVec
    \end{bmatrix}\fT^{-1}\stx(0^{+})\oplus-\sum_{j=0}^{\indDAE-1} \fN^j\source^{(j)}_{\stx^{\imp}}(0^{-})\right),
\end{equation}
leaving the solution at time $t=0$ discontinuous. Therefore, distributional solutions are necessary to fully characterize the existence and uniqueness of \IVP~\eqref{eqn:DAE}. The space of a piecewise-smooth distribution, denoted as $\mathbb{D}_{pw\calC^{\infty}}$
(see \cite{Tre10}) is suitable to prove the existence and uniqueness result for the solution of \eqref{eqn:DAE}.
\begin{theorem}(Existence and uniqueness of \IVP solutions \cite[Thm.~6.5.1]{Tre12})
Consider the \IVP~\eqref{eqn:DAE} with regular matrix $(\fE,\fA)$. Then, for all initial trajectories $\stx_0 \in \mathbb{D}^{\stateDim}_{pw\calC^{\infty}}$ and all source function $\source \in \mathbb{D}^{\stateDim}_{pw\calC^{\infty}}$, there exist a unique $\stx \in \mathbb{D}^{\stateDim}_{pw\calC^{\infty}}$ that satisfies the \IVP~\eqref{eqn:DAE}.
\end{theorem}
To have a continuous solution it is necessary that $\stx(0^+)=\stx(0^-)$, implying that \eqref{eqn:ini:sol} imposes a so-called \emph{consistency condition} on the initial value $\stx_0$ for a classical solution to exist; see \cite[Cha.~2]{KunM06}. Let us observe that the matrix 
\begin{align}\label{eqn:diff:pro}
    \begin{aligned}
            \fPi\;\vcentcolon=\;\fT\begin{bmatrix}
                \fI_{\stateDim_{\fJ}}&\zeroMat\\
                \zeroMat&\zeroMat
            \end{bmatrix}\fT^{-1}
    \end{aligned}
\end{align}
is a projector on the differential space associated with the \DAE.  The matrices~$\fS,\fT$ can be constructed using the \emph{Wong sequences} \cite{Won74}, which are defined as 
\begin{subequations}
	\label{eqn:WongSequences}
	\begin{align}
		\calV^0 &\coloneqq \R^n, & \calV^{i+1} &\coloneqq \fA^{-1}(\fE\calV^i), &  i\in\N,\\
		\calW^0 &\coloneqq \{0\}, & \calW^{j+1} &\coloneqq \fE^{-1}(\fA\calW^j), & j\in\N,\label{eqn:WongSequencesImp}
	\end{align}
\end{subequations}
where we use the notation for the preimage as in~\eqref{eqn:pre:img}. After finitely many steps the sequences in \eqref{eqn:WongSequences} converge and the limits are given by
\begin{equation}
	\label{eqn:WongSequences:limits}
	\calV^\star \coloneqq \bigcap_{i\in\N} \calV^i \qquad\text{and}\qquad \calW^\star \coloneqq \bigcup_{i\in\N} \calW^i.
\end{equation}
	
\begin{theorem}[{\QWF via Wong sequences, \cite[Thm.~2.6]{BerIT12}}]
	\label{thm:QWF:Wong}  
	Consider a regular matrix pair $(\fE,\fA)$ with corresponding Wong limits $\calV^\star$ and $\calW^\star$. For any full rank matrices~$\difW$ and $\impW$ such that $\img(\difW)= \calV^\star$ and $\img(\impW) = \calW^\star$, the matrices 
	\begin{equation}\label{eqn:Wong:mat}
		\fT = [\difW, \impW],\quad \fS = [\fE\difW, \fA\impW]^{-1}
	\end{equation}
	are invertible and transform $(\fE,\fA)$ into \QWF~\eqref{eqn:QWF}.
\end{theorem}

\subsection{Are \CIM suitable for \DAE?}\label{subsec:CIM-DAE}
We start by characterizing the behavior of $\|(\lapVar\fE-\fA)^{-1}\|$ in the complex plane. We first review the standard \ODE case where $\fE=\fI_{\stateDim}$, then discuss the matrix pencil arising from the \DAE setting, and finally provide conditions under which the \CIM approximation detailed in \Cref{subsec:CIM} is well suited for the \DAE setting.

\subsubsection{The \ODE case}
Let us recall some standard knowledge about the case $\fE=\Id_{\stateDim}$, i.e., when \eqref{eqn:DAE} is a linear time-invariant system of \ODE. In such a setting, the magnitude of the resolvent norm $\|(\lapVar\Id_{\stateDim}- \fA)^{-1}\|$ plays a crucial role in the convergence rate of any contour integral method based on the Laplace transformation; see \cite{GugLM21}. Due to this, the choice and parameterization of the integration contour are of major importance and need to account for the magnitude of the resolvent norm. This choice is made in \cite{GugLM21} through knowledge of the $\varepsilon$-pseudospectrum of $\fA$ (see \cite{TreE05}), which can be defined as
\begin{equation}\label{eqn:eps:pseudo}
    \sigma_{\varepsilon}(\fA)\;\vcentcolon=\;\left\{\lapVar\in\C\;|\;\|(\lapVar\Id_{\stateDim}- \fA)^{-1}\|\;\ge\;\frac{1}{\varepsilon} \right\},
\end{equation}
for suitable values of $\varepsilon>0$. If $\fA$ is normal, then \eqref{eqn:eps:pseudo} can be completely characterized in terms of the distance of $\lapVar$ from the spectrum of $\fA$. However, since $\fA$ is in general nonnormal, for example, in the case where $\fA$ includes the discretization of a convective term, $\|(\lapVar\Id_{\stateDim}- \fA)^{-1}\|$ may be large even when $\lapVar$ is not close to the spectrum of $\fA$. Since the eigenvalues of $\fA$ are finite, it is immediate that the asymptotic behavior of $\|(\lapVar\Id_{\stateDim}- \fA)^{-1}\|$ is the same as $|z|^{-1}$. For this reason, in the context of \ODE, \CIMs are well suited whenever 
\[
\lim_{|\lapVar|\rightarrow\infty}\frac{\hat{\source}(\lapVar)}{|\lapVar|^{\alpha}}\;\lutc \;1,\quad \text{with}\quad \alpha<1.
\]

We refer to the method proposed in \cite{GugLM21} for the construction of the integration profile $\Gamma$ guided by the   evaluation of the resolvent norm in suitable points. Recently, an algorithm that approximates the resolvent norm over a compact subset of the complex plane was proposed in \cite[Sec.~5]{ManMG26}. The method is based on the use of the subspace approach and in the interpretation of the frequency variable $\lapVar$ as a two-dimensional parameter. The same idea, but with a different subspace approach, was previously presented in \cite{Sir19}. 

\subsubsection{Asymptotic and not asymptotic behavior of \texorpdfstring{$(\lapVar\fE-\fA)^{-1}$}{TEXT}}
The norm of the matrix pencil $(\lapVar\fE-\fA)^{-1}$ has been studied in the context of transient growth bound for systems of \DAE, see \cite{EmbK17}. With this aim, the authors propose a definition of the $\varepsilon$-pseudospectrum for the matrix pair $(\fE,\fA)$. In our contest, we are not interested in the transient growth of \DAE, but rather in the location of the singularities of the rational function $\|(\lapVar\fE-\fA)^{-1}\|$ and in the characterization of the asymptotic behavior of $\|(\lapVar\fE-\fA)^{-1}\|$ as well as its behavior in specific regions of the complex plane. The following proposition provides such a characterization. Such a result is certainly known, although we did not find an explicit reference in the literature; we therefore state it here for the sake of completeness.

\begin{proposition}\label{lemma1}
    Suppose that the matrix pair $(\fE,\fA)$ is regular and consider its quasi-Weierstrass form via the matrices $\fS,\fT,\fJ,$ and $\fN$ as described in \Cref{subsec:DAE}. Then, for any $\lapVar\in\C$, one has
    \begin{equation}\label{eqn:bound:psEA}
        \|(\lapVar\fE-\fA)^{-1}\|\;=\;\max\left(\|(\lapVar\Id_{\stateDim_{\fJ}}-\fJ)^{-1}\|,\Bigg\|\sum_{j=0}^{\indDAE-1}\lapVar^j\fN^j\Bigg\|\right).
    \end{equation}
\end{proposition}
\begin{proof}
    Since the matrix pair $(\fE,\fA)$ is assumed to be regular, there exist two invertible square matrices $\fS,\fT$ that allow us to decouple the system using \QWF~\eqref{eqn:QWF}. It is not restrictive to assume that these matrices are also orthonormal. Thus, we have
    \begin{align}
    \begin{aligned}
        \|(\lapVar\fE-\fA)^{-1}\| \; =&\;\|\fT^{-1}(\lapVar\fS\fE\fT-\fS\fA\fT)^{-1}\fS^{-1}\|\;\\
        =&\;\Bigg\|\fT^{-1}\begin{bmatrix}
            (\lapVar\Id_{\stateDim_{\fJ}}-\fJ)^{-1}&\zeroMat\\
            \zeroMat&(\lapVar\fN-\Id_{\stateDim_{\fN}})^{-1}
        \end{bmatrix}\fS^{-1}\Bigg\|\;\\
        =&\;\max\left(\|(\lapVar\Id_{\stateDim_{\fJ}}-\fJ)^{-1}\|,\|(\lapVar\fN-\Id_{\stateDim_{\fN}})^{-1}\|\right),
        \end{aligned}
    \end{align}
    where we used the fact that the spectral norm is invariant under unitary transformations and that the singular values of block diagonal matrix are the union of the singular values of each block. Regarding $(\lapVar\fN-\fI)^{-1}$ we can use
    \begin{equation}
        (\lapVar\fN-\fI_{\stateDim_{\fN}})^{-1}\;=\;-\sum_{j=0}^{\indDAE-1}\lapVar^j\fN^j\;=\;\Id_{\stateDim_{\fN}}+\lapVar\fN+\dots+\lapVar^{\indDAE-1}\fN^{\indDAE-1},
    \end{equation}
    from which follows
    \begin{equation*}
         \|(\lapVar\fE-\fA)^{-1}\|\;=\;\max\left(\|(\lapVar\Id_{\stateDim_{\fJ}}-\fJ)^{-1}\|,\Bigg\|\sum_{j=0}^{\indDAE-1}\lapVar^j\fN^j\Bigg\|\right).
    \end{equation*}
\end{proof}
 
    From the proof of \Cref{lemma1} it is clear that the singularities of $(\lapVar\fE-\fA)^{-1}$ coincide with the eigenvalues of $\fJ$, i.e., with the finite eigenvalues of the matrix pair $(\fE,\fA)$. In fact, $\lapVar\fN-\Id_{\stateDim_{\fN}}$ is invertible $\forall \lapVar\in\C$ as $\fN$ is a nilpotent matrix. The behavior of $\|(\lapVar \Id_{\stateDim_{\fJ}}-\fJ)^{-1}\|$ is the same as that of a resolvent of a standard Hurwitz matrix ($\fJ$ is asymptotically stable by assumption); thus, in compact regions of the complex plane where $|\lapVar|$ is sufficiently small, $\|(\lapVar \fE-\fA)^{-1}\|$ behaves as the resolvent of the matrix $\fJ$. This implies that standard pseudospectra routines may be suitable for the identification of level sets of $\|(\lapVar\fE-\fA)^{-1}\|$ in certain regions of the complex plane. 
    
    The issue with the term $\lapVar\fN-\Id_{\stateDim_{\fN}}$ is that, asymptotically, it makes $\|(\lapVar\fE-\fA)^{-1}\|$ unbounded; thus, at first glance, it does not allow \cref{ass:2} to be generically satisfied and also generates singularity points at infinity. However, as we show in the next subsection, under standard conditions for the time integration of \DAE, we can guarantee that \cref{ass:1} and \cref{ass:2} are satisfied.
    
     \subsubsection{Conditions for the \CIM applied to \DAE}

    \begin{theorem}\label{lemma2}
        Assume that the matrix pair $(\fE,\fA)$ is regular and that $\source_{{\diff}}$ and $\source_{{\imp}}$, as defined in \eqref{eqn:DAE:slow} and \eqref{eqn:DAE:fast}, respectively, fulfill the conditions
         \begin{subequations}\label{eqn:dec:ass}
             \begin{align}
                \lim_{|\lapVar|\rightarrow\infty} \frac{\|\hat \source_{{\diff}}(\lapVar)\|}{|\lapVar|^{\alpha}}\;&\lutc\;1,\label{eqn:ass:dif:source}\\
            \lim_{|\lapVar|\rightarrow\infty} \|\hat \source^{(j)}_{{\imp}}(\lapVar)\|\;&\rightarrow\;0,\quad j=0,\ldots,\indDAE-1;\label{eqn:ass:imp:source}
             \end{align}
         \end{subequations}
         with $\alpha<1$ and $\hat \source^{(j)}_{{\imp}}$ denoting the Laplace transform of the $j$st time derivative of $\source^{(j)}_{{\imp}}$. Then the integrand in \eqref{eqn:inv:Lap} is decaying, i.e. 
         \begin{equation}\label{lemma2:statement}
             \lim_{|\omega|\rightarrow\infty} \|\ee^{\gamma+\imagunit\omega}\hat{\stx}(\gamma+\imagunit\omega)\|\rightarrow0.
         \end{equation}
         \end{theorem}
         \begin{proof}
             To prove the statement, due to the boundedness of the exponential over vertical lines in the complex plane, it is sufficient to show that \eqref{eqn:lap} goes to zero as ${|\lapVar|\rightarrow\infty}$. Consider \eqref{eqn:DAE}, where we take the change of variable $\tilde{\stx}\vcentcolon=\fT^{-1}\stx$ and multiply the \DAE from the left by $\fS$, assuming both $\fS$ and $\fT$ are orthonormal we get the following system
\begin{equation}
	\label{eqn:new:DAE}
	\left\{\quad \begin{aligned}
		\begin{bmatrix}
		    \Id_{\stateDim_{\fJ}}&\zeroMat\\
            \zeroMat&\fN
		\end{bmatrix} \dot{\tilde\stx}(t)  &= \begin{bmatrix}
		    \fJ&\zeroMat\\
            \zeroMat&\Id_{\stateDim_{\fN}}
		\end{bmatrix}\tilde \stx(t) +\fS \source(t),\\  
        \tilde\stx(t_0) &= \fT^{-1}\inSol
	\end{aligned}\right.
\end{equation} 
Now, taking the Laplace transform in \eqref{eqn:new:DAE} we find
\begin{equation*}
   \hat{\tilde\stx}(\lapVar)\;=\;\left(\begin{bmatrix}
        \left(\lapVar\Id_{\stateDim_{\fJ}}-\fJ\right)^{-1}&\zeroMat\\
        \zeroMat&\zeroMat 
    \end{bmatrix}+\begin{bmatrix}
        \zeroMat&\zeroMat\\
        \zeroMat&\left(\lapVar\fN-\fI_{\stateDim_{\fN}}\right)^{-1}
    \end{bmatrix}\right)\left(\begin{bmatrix}
		    \fI_{_{\stateDim_{\fJ}}}&\zeroMat\\
            \zeroMat&\fN
		\end{bmatrix}\tilde\stx(0^{-})+\fS\hat \source(\lapVar)\right)
\end{equation*}
and we immediately note that, by using the decoupling \eqref{eqn:DAE:slow}-\eqref{eqn:DAE:fast}
\begin{equation}\label{eqn:imp:5}
    \|\hat{\stx}(\lapVar)\|\;=\;\|\hat{\tilde \stx}(\lapVar)\|\;\le \;\|\hat{\stx}_{\diff}(\lapVar)\|+\|\hat{\stx}_{\imp}(\lapVar)\|,
\end{equation}
therefore, it suffices to independently evaluate the asymptotic behavior of the norm of the Laplace transform associated with the slow subsystem state and that of the fast subsystem state. Recalling the relation \eqref{eqn:ini:sol}, for $\hat{\stx}_{\diff}(\lapVar)$ we have
\begin{align}\label{eqn:dif:con}
    \begin{aligned}
        \lim_{|\lapVar|\rightarrow\infty}\|\hat{\stx}_{\diff}(\lapVar)\|\;=&\; \lim_{|\lapVar|\rightarrow\infty}\Bigg\|
        \left(\lapVar\fI_{\stateDim_{\fJ}}-\fJ\right)^{-1} \left(\begin{bmatrix} \fI_{\stateDim_{\fJ}}&\zeroVec
    \end{bmatrix}\fT^{-1}\stx(0^{-})+\hat \source_{\diff}(\lapVar)\right)\Bigg\|\;\\
    \lutc&\; \lim_{|\lapVar|\rightarrow\infty}\left( |\lapVar|^{-1}+|\lapVar|^{\alpha-1}\right)\rightarrow 0,
    \end{aligned}
\end{align}
where we used \eqref{eqn:ass:dif:source} and the fact that $\|(\lapVar\fI_{\stateDim_{\fJ}}-\fJ)^{-1}\|$ behaves asymptotically as the resolvent norm in the \ODE case. For $\|\hat{\stx}_{\imp}(\lapVar)\|$ we have
\begin{equation}\label{eqn:imp:1}
   \fN\hat{\stx}_{\imp}(\lapVar)\;=\;\frac{1}{\lapVar}\left(\fN{\stx}_{\imp}(0^{-})+\hat{\stx}_{\imp}(\lapVar)+\hat{\source}_{\imp}(\lapVar)\right),
\end{equation}
which gives, after multiplication by $\fN^{\indDAE-1}$, the expression
\begin{equation}\label{eqn:imp:2}
  \fN^{\indDAE-1}\hat{\stx}_{\imp}(\lapVar)+\fN^{\indDAE-1}\hat{\source}_{\imp}(\lapVar)\;=\;\zeroVec.
\end{equation}
Substituting $\indDAE-1$ times \eqref{eqn:imp:1} into \eqref{eqn:imp:2} we find
\begin{equation}\label{eqn:imp:3}
  \hat{\stx}_{\imp}(\lapVar)\;=\;-\sum_{j=0}^{\indDAE-1}\fN^{j}\lapVar^{j}\hat{\source}_{\imp}(\lapVar)-\sum_{j=1}^{\indDAE-1}\lapVar^{j-1}\fN^{j}\stx_{\imp}(0^{-}).
\end{equation}
Recalling \eqref{eqn:sol:sub:fast} and the properties of the Laplace transform for derivatives, i.e., 
\begin{equation}\label{eqn:imp:4}
    \stx_{\imp}(0^{-})\;=\;-\sum_{i=0}^{\indDAE-1} \fN^i\source^{(i)}_{{\imp}}(0^{-}),\quad \quad\lapVar^{j}\hat{\source}_{\imp}(\lapVar)\;=\;\hat{\source}^{(j)}_{\imp}(\lapVar)+\sum_{i=1}^{j}\lapVar^{j-i}\source_{\imp}^{(i-1)}(0^{-}),
\end{equation}
and substituting \eqref{eqn:imp:4} into \eqref{eqn:imp:3}, after a suitable reorder of the indexes, we obtain 
\begin{equation*}
    \hat{\stx}_{\imp}(\lapVar)\;=\;-\sum_{j=0}^{\indDAE-1}\fN^{j}\hat{\source}^{(j)}_{\imp}(\lapVar).
\end{equation*}
Finally, using assumption \eqref{eqn:ass:imp:source} 
\[
\lim_{|\lapVar|\rightarrow\infty}\|\hat{\stx}_{\imp}(\lapVar)\|\;=\;\lim_{|\lapVar|\rightarrow\infty}\Bigg\|\sum_{j=0}^{\indDAE-1}\fN^{j}\hat{\source}^{(j)}_{\imp}(\lapVar)\Bigg\|\;\lutc\; \lim_{|\lapVar|\rightarrow\infty}\left(\sum_{j=0}^{\indDAE-1}\|\fN^{j}\|\|\hat{\source}^{(j)}_{\imp}(\lapVar)\|\right)\;\rightarrow\;0;
\]
which, together with \eqref{eqn:imp:5} and \eqref{eqn:dif:con}, concludes the proof.

\end{proof}
\Cref{lemma2} shows that the condition imposed on the differential component of the source term coincides with the one required in the \ODE case. In contrast, for the impulsive subsystem, the assumption \eqref{eqn:ass:imp:source} is automatically fulfilled by any source term that is sufficiently regular in time, which aligns with the usual requirements for classical solutions implied by \eqref{eqn:sol:sub:fast}. Indeed, invoking the causality principle as in \cite[Sec.~2.3]{Col22}), one readily verifies that if $\source^{(j)}$ is continuous on $[0,T]$, then its Laplace transform decays for $|\lapVar|\rightarrow\infty$ as $\calO(|z|^{-1})$. Hence, \eqref{eqn:dec:ass} holds for every function that is $\indDAE-1$ times differentiable on the interval of interest $[0,T]$. Furthermore, functions with mild temporal singularities, such as $1/\sqrt{t-\alpha}$ for $\alpha\in[0,T]$, are also admissible since their Laplace transforms exhibit decay as well.


\section{Uniform bound for the parametric generalized resolvent 
via structured-unstructured perturbations}\label{sec:str-uns}

Consider the parametric dynamical system of the form
\begin{equation}
	\label{eqn:parametric:DAE}
	\left\{\quad \begin{aligned}
		\fE \dot{\stx}(t;\prmtr)  &= \fA(\prmtr)\stx(t;\prmtr) +\source(t),\\
        \stx(t_0) &= \inSol
	\end{aligned}\right.
\end{equation}
where $\prmtr \in \prmtrSet$, with $\prmtrSet$ a compact subset (possibly unknown) of $\R^{\prmtrDim}$, and $\fA:\prmtrSet\rightarrow\C^{\stateDim\times\stateDim}$ is a matrix-valued function depending on the parameter and having the structure
\begin{align}
\begin{aligned}\label{eqn:CS:par:mat}
        \fA(\prmtr)\;=\;\sum_{j=1}^{Q_{\fA}}\alpha_j(\prmtr)\fA_j,
\end{aligned}
\end{align}
for some analytic functions $\alpha_j:\prmtrSet\rightarrow\R$, constant matrices $\fA_j\in\R^{\stateDim\times\stateDim}$, and a positive integer $Q_{\fA}\ll\stateDim$. For clarity of exposition, we restrict attention here to the non-parametric case of $\fE$. The corresponding parametric setting will be discussed in subsequent subsections.

We are interested in approximating $\stx(t;\prmtr)$ for $t\in[T,\Lambda T]$, with $\Lambda>1$, by means of \CIM. In particular, since computing the contour $\Gamma$ can be computationally expensive, our goal is to determine the integration profile for only one (or a few) parameter value(s) $\prmtr_0\in\prmtrSet$, and then reuse the corresponding profile $\Gamma_{\prmtr_0}$ for all other choices of $\prmtr$. The ability to accurately integrate system \eqref{eqn:parametric:DAE} using only a single profile (or a small number of) is essential for the efficiency of \CIM in parametric settings, especially for its applications in projection-based \MOR; see \cite{GugM23}. There are two points that require attention when using the profile $\Gamma_{\prmtr_0}$ for parameter values different from $\prmtr_0$:
\begin{enumerate}
    \item the generalized eigenvalues of the matrix pair $(\fE,\fA(\prmtr))$ are different from those evaluated at $\prmtr_0$. To efficiently apply \CIM, \cref{ass:1} must be satisfied, which implies that $\Gamma_{\prmtr_0}$ has to enclose the eigenvalues of $(\fE,\fA(\prmtr))$ for all $\prmtr\in\prmtrSet$;\label{issue1}

    \item the magnitude of the resolvent $\|(\lapVar(\intvar)\fE-\fA(\prmtr))^{-1}\|$ is fundamental for bounding the approximation error of \CIM; see \cite{GugLN18}. In \cite{GugLM21}, the contour $\Gamma_{\prmtr_0}$ is chosen with the goal of ensuring that $\|(\lapVar(\intvar)\fE-\fA(\prmtr_0))^{-1}\|$ remains below a prescribed threshold for all $\lapVar\in\Gamma_{\prmtr_0}$. Nevertheless, the resolvent norm may vary significantly with the parameter $\prmtr$, in particular when the parameter induces strong non-normality in the operator or moves some eigenvalues towards $\Gamma_{\prmtr_0}$.\label{issue2}
\end{enumerate}
Both \cref{issue1} and \cref{issue2} must be resolved to enable a rigorous application of any \CIM method in the parametric framework.

As a starting point, in this work, we address \cref{issue2} by adopting the two-level viewpoint of structured matrix nearness problems, specialized here in the control of the generalized resolvent under structured perturbations. More precisely, the resolvent bound is first reformulated as a structured distance-to-singularity problem, and then solved numerically through a joint structured--unstructured optimization procedure, in the spirit of \cite[Alg.~8]{GL24}. This formulation is consistent with the joint structured--unstructured pseudospectrum and the structured $\varepsilon$-stability radius viewpoint developed in the book framework.

Suppose that for a certain parameter $\prmtr_0$ we construct the integration profile $\Gamma_{\prmtr_0}$ to apply the \CIM described in \cite{GugLM21}. Moreover, assume that this is done in such a way that, for a prescribed $\varepsilon$, we have
\[
\|(\lapVar(\intvar_j)\fE-\fA(\prmtr_0))^{-1}\|\le\frac{1}{\varepsilon}
\qquad \text{for all quadrature points } j=1,\ldots,N-1.
\]
Our goal here is the following: determine the set $\hat \prmtrSet\subseteq\prmtrSet$ such that, for all $j=1,\ldots,N-1$, it holds
\begin{equation*}
    \|(\lapVar(\intvar_j)\fE-\fA(\prmtr))^{-1}\|\;\le\;\frac{1}{\varepsilon},\quad\text{for all }\prmtr\in\hat\prmtrSet.
\end{equation*}
Controlling the resolvent magnitude uniformly over all the quadrature points and across all the parameters is essential to handle the numerical error; see the discussion in \cite[Sec.~3.5]{GugM23}. In the next subsections, we detail two approaches to determine $\hat\prmtrSet$. 

\subsection{A singular value optimization problem}\label{sec:problem}
Let $(\fE,\fA_0)$ be the matrix pencil associated with a given $\prmtr_0\in\prmtrSet$ such that
\[
\|(\lapVar \fE-\fA_0)^{-1}\|<\frac{1}{\varepsilon}
\]
for a prescribed $\varepsilon>0$. Due to the relation $\|\fM^{-1}\|=\sigma_{\min}(\fM)^{-1}$, the condition
$
\|(\lapVar\fE-\fA_0)^{-1}\|<{\varepsilon^{-1}}
$
is equivalent to $\sigma_{\min}(\lapVar\fE-\fA_0)>\varepsilon$. For $j=1,\ldots,Q_{\fA}$, we denote by $\cS_j$ the linear space
\begin{align}\label{eqn:manifold}
    \cS_{\fA_j}\;\vcentcolon=\;\left\{\fM\;\left|\,
    \fM=-g_j\fA_j,\ \text{with }g_j\in\R,\, \fA_j \text{ the $j$-st term in the sum \eqref{eqn:CS:par:mat}}
    \right.\right\}.
\end{align}
The problem we want to solve is the following: find the largest $\delta>0$ such that, for every $\fDelta_j\in\cS_{\fA_j}$, with $\|\fDelta_j\|_{\Frob}\le\delta$, and with $j=1,\ldots,Q_{\fA}$, we have
\begin{equation}\label{eqn:direct-sv}
    \sigma_{\min}\left(\lapVar\fE-\fA_0+\sum_{j=1}^{Q_{\fA}}\fDelta_j\right)\ge\varepsilon,
\end{equation}
where $\fDelta_j\in\cS_{\fA_j}$.
In most cases, this is equivalent to looking for the smallest $\delta>0$ such that equality is maintained in \eqref{eqn:direct-sv}. 

Let us define $\cS_{\fA}\vcentcolon=\Pi_{j=1}^{Q_{\fA}}\cS_{\fA_j}$ equipped with the norm
\begin{equation*}
\|\fDelta\|_{\cS_{\fA}}=\max_{j=1,\ldots, Q_{\fA}}\|\fDelta_j\|_{\Frob},
\end{equation*}
where $\fDelta\vcentcolon=[\fDelta_1,\ldots,\fDelta_{Q_{\fA}}]$ and $\fDelta_{j}\in\cS_{\fA_j}$. We say that $\delta$ is the $\cS_{\fA}$-structured $\varepsilon$-distance to singularity of $\lapVar\fE-\fA_0$.

For subsequent algorithmic development, it is convenient to use the structured–unstructured formulation with normalized matrices: find the smallest $\delta>0$ such that there exist $\tilde \fL\in\cS_{\fA}$ with $\tilde \fL=[\tilde \fL_1,\ldots,\tilde \fL_{Q_{\fA}}]$, $\tilde \fL_j\in\cS_{\fA_j}$, $\|\tilde \fL_j\|_{\Frob}=1$, and $\fL\in\fC^{\stateDim\times\stateDim}$ of rank one with $\|\fL\|_{\Frob}=1$, such that
\begin{equation}\label{eqn:joint-sing}
    \lapVar\fE-\fA_0+\delta\sum_{j=1}^{Q_{\fA}}\tilde \fL_j+\varepsilon\fL
\end{equation}
is singular. 
\begin{remark}
The equivalence between \eqref{eqn:direct-sv} and \eqref{eqn:joint-sing} follows from the Eckart--Young characterization of the distance to singularity in the Frobenius norm. In particular, the unstructured perturbation can be chosen to be of rank one.
\end{remark}

\begin{remark}
This formulation is consistent with the viewpoint of structured matrix nearness problems and structured $\varepsilon$-stability radii: the structured perturbation $\fDelta$ controls the admissible parameter variation, while the unstructured rank-one perturbation $\fP$ realizes the singularity threshold corresponding to the prescribed resolvent level $\varepsilon$.
\end{remark}

\begin{remark}[Extension to the case of $\fE$ being a parametric function]
    To formally deal with $\fE$ being a parametric matrix-valued function, i.e. $\fE:\prmtrSet\rightarrow\R^{\stateDim\times\stateDim}$, we also assume an affine dependent structure
    \begin{equation}\label{eqn:CS:par:mat:E}
        \fE(\prmtr)\;=\;\sum_{j=1}^{Q_{\fE}}\epsilon_j(\prmtr)\fE_j,
    \end{equation}
    with analytic functions $\epsilon_j:\prmtrSet\rightarrow\R$. Then we proceed by defining $\cS_{\fE_j}$, for $j=1,
    \ldots, Q_{\fE}$, as
     \begin{align}\label{eqn:manifold:E}
    \cS_{\fE_j}\;\vcentcolon=&\;\left\{\fM\;\left|\,
    \fM=g_j \fE_j,\  \text{with }g_j\in\R,\, \fE_j \text{ the $j$-st term in the sum \eqref{eqn:CS:par:mat:E}}   \right.\right\},
    \end{align}
    and the set $\cS_{\fE}$ as $\cS_{\fE}\vcentcolon=\Pi_{j=1}^{Q_{\fE}}\cS_{\fE_j}$. For $\fDelta=[\fDelta_1,\ldots,\fDelta_{Q_{\fE}}]\in\cS_{\fE}$ we define $\|\fDelta\|_{\cS_{\fE}}$ as 
    \begin{equation*}
        \|\fDelta\|_{\cS_{\fE}}\;\vcentcolon=\;\max_{j=1,\ldots,Q_{\fE}}\|\fDelta_j\|_{\Frob},\quad\fDelta_j\in\cS_{\fE_j}.
    \end{equation*}
    Let $(\fE_0,\fA_0)$ be the matrix pencil associated with a given $\prmtr_0\in\prmtrSet$ such that
    \[
    \|(\lapVar \fE_0-\fA_0)^{-1}\|<\frac{1}{\varepsilon}.
    \]
    Then, we define the following problem: find the largest $\delta>0$ such that for every $\fDelta_{\fA}\in\cS_{\fA}$ and $\fDelta_{\fE}\in\cS_{\fE}$ with $\|\fDelta_{\fA}\|_{\cS_{\fA}}\le\delta$ and $\|\fDelta_{\fE}\|_{\cS_{\fE}}\le\delta$, we have
    \begin{equation*}
    \sigma_{\min}\left(\lapVar\fE_0-\fA_0+\lapVar\sum_{j=1}^{Q_{\fE}}\fDelta_{\fE_j}+\sum_{j=1}^{Q_{\fA}}\fDelta_{\fA_{j}}\right)\;\ge\;\varepsilon.
    \end{equation*}
\end{remark}

\subsection{Alternative reformulation}

The previous formulation is the most natural one from the point of view of the structured distance to singularity. However, in some situations, the parameter dependence can be exploited more directly.

Let a fixed $\eps>0$ and $z \in \mathbb{C}$ be given. Suppose that one restricts the admissible parameter variation to a one-dimensional path issued from $\prmtr_0$, for instance, of the form $\prmtr=(1+\delta)\prmtr_0$ when this is meaningful in the parameter domain. Then, for varying $\delta$, one may introduce the functional
\begin{equation}\label{eqn:F-delta-alt}
F_{\varepsilon}(\fL,\delta) = \sigma_{\min}\!\left(\lapVar\fE-\sum_{i=1}^{Q_{\fA}}\alpha_i((1+\delta)\prmtr_0)\fA_i+ \eps \fL\right),
\end{equation}
for $\fL\in\C^{\stateDim\times\stateDim}$ of unit Frobenius norm. This leads to the classical two-level viewpoint:
\begin{itemize}
    \item \textbf{Inner iteration:} for a given $\delta>0$, one computes only the unstructured perturbation matrix $\fL(\delta)$, that is, computing
		\[
		\fL(\delta) =
      \arg\min\limits_{\|\fL\|_{\Frob}=1} 
      \sigma_{\min}\!\left(\lapVar\fE-\sum_{j=1}^{Q_{\fA}}\alpha_j((1+\delta)\prmtr_0)\fA_j+\eps\fL\right).
		\]
    \item \textbf{Outer iteration:} one computes the smallest positive value $\delta_{\varepsilon}$ such that
    \[
      \ophi(\delta_\eps)\;=\; 0, \qquad 		\ophi(\delta)\;\vcentcolon=\; \sigma_{\min}\!\left(\lapVar\fE-\sum_{j=1}^{Q_{\fA}}\alpha_j((1+\delta)\prmtr_0)\fA_j+\eps\fL(\delta)\right).
    \]
\end{itemize}

A further possibility is to parameterize the variation by an additive increment $\fdelta$ and consider $\prmtr=\prmtr_0+\fdelta$.
The corresponding outer problem is then finite-dimensional but nonlinear in the parameter increment, and it naturally suggests gradient-based methods with line-search or backtracking. We do not pursue this alternative approach further here but mention it as a possible complementary strategy.

\subsection{Two-level iteration}\label{sec:inn:out:alg}
Our numerical method for addressing the problem described in \Cref{sec:problem} is based on a two-level iterative algorithm, similar in spirit to \cite[Alg.~8]{GL24}.

Let a fixed $\eps>0$ be given. For variable $\delta>0$, we introduce the functional
\begin{equation}\label{F-delta}
F_{\varepsilon}(\tilde \fL,\fL,\delta)\;
=\;
\sigma_{\min}\left(\lapVar\fE-\fA_0+\delta \sum_{j=1}^{Q_{\fA}}\tilde \fL_j + \eps \fL\right),
\end{equation}
for $\tilde \fL=[\tilde \fL_1,\ldots,\tilde \fL_{Q_{\fA}}]\in \cS_{\fA}$, $\tilde \fL_j\in\cS_{\fA_j}$, and $\fL \in \C^{\stateDim,\stateDim}$, both of unit Frobenius norm. With this functional, we follow a two-level approach:
\begin{itemize}
\item {\bf Inner iteration:\/} For a given $\delta>0$, we aim to compute matrices $\tilde \fL(\delta) \in\cS_{\fA}$ and $\fL(\delta) \in \C^{\stateDim,\stateDim}$, both of unit Frobenius norm, that minimize $F_\varepsilon$:
\begin{equation} \label{L-delta}
(\tilde \fL(\delta),\fL(\delta))
\;=
\;\arg\min_{\substack{
\tilde\fL_j \in \cS_{\fA_j},\ \fL\in \C^{\stateDim,\stateDim} \\
\| \tilde \fL_j\|_\Frob=1,\ \| \fL \|_\Frob = 1 \\
j=1,\ldots,Q_{\fA}
}}
F_\varepsilon(\tilde \fL,\fL,\delta).
\end{equation}

\item {\bf Outer iteration:\/} We compute the smallest positive value $\delta_\eps$ such that
\begin{equation} \label{zero-delta}
\ophi(\delta_\eps)\;=\; 0,\qquad
\text{with}\qquad
\ophi(\delta)\;\vcentcolon=\;
F_\eps\bigl(\tilde \fL(\delta),\fL(\delta),\delta\bigr).
\end{equation}
\end{itemize}

Provided that these computations succeed, we obtain a structured perturbation
\[
\fDelta_\eps= \delta_\eps \sum_{j=1}^{Q_{\fA}}\tilde\fL_j(\delta_\eps) 
\]
such that
\[
\sigma_{\min}(\lapVar\fE-\fA_0+\fDelta_\eps)=\eps,
\]
that is, $\delta_\eps$ is an approximation of the $\cS_{\fA}$-structured $\eps$-distance to singularity of $\lapVar\fE-\fA_0$. As in the general two-level framework, the computed quantity should be regarded as an upper bound when the inner iteration converges only to a local minimum.

\subsubsection{Matrix ODEs for the inner iteration} \label{subsec:r1-ode-eps-rad}

In this subsection, we detail a possible way to solve \eqref{L-delta}. Before doing so, let us recall a standard result on the derivative of simple singular values.

\begin{theorem}[see, for instance, Lemma 1 in \cite{GugLM21}]
\label{lem:eigderiv}
Consider a continuously differentiable path of square complex matrices $\fM(t)$ for $t$ in an open interval $I$. Let $\sigma(t)$, $t\in I$, be a continuous path of simple positive singular values of $\fM(t)$. Let $\fu(t)$ and $\fv(t)$ be the associated left and right singular vectors, respectively, i.e.,
\[
\fM(t)\fv(t)=\sigma(t)\fu(t),\qquad \fM(t)^*\fu(t)=\sigma(t)\fv(t).
\]
Then $\sigma$ is continuously differentiable on $I$ and its derivative is given by
\begin{equation*}
\dot{\sigma}(t)
=
\Re\left(\fu(t)^* \dot{\fM}(t) \fv(t)\right)
=
\Re\big\langle \fu(t)\fv(t)^*, \dot \fM(t) \big\rangle .
\end{equation*}
\end{theorem}

\noindent
The following result allows us to compute the steepest descent direction of the functional $F_\eps$.

\begin{lemma}[Free gradient]
\label{lem:gradient}
Let $\tilde \fL(t)=[\tilde \fL_1,\ldots,\tilde \fL_{Q_{\fA}}]\in \cS_{\fA}$, with $\tilde \fL_j\in\cS_{\fA_j}$ for $j=1,\ldots,Q_{\fA}$, and $\fL(t)\in \C^{\stateDim,\stateDim}$, for real $t$ near $t_0$, be continuously differentiable paths of matrices, with derivatives denoted by $\dot {\tilde \fL}(t)$ and $\dot \fL(t)$.
Assume that $\sigma(t)$ is a simple singular value of
\[
\lapVar\fE-\fA_0+\delta\sum_{j=1}^{Q_{\fA}}\tilde \fL_j(t)+\varepsilon\fL(t)
\]
that depends continuously on $t$, with associated left and right singular vectors $\fu(t)$ and $\fv(t)$. Then
\[
F_\varepsilon(\tilde \fL(t),\fL(t),\delta)= \sigma_{\min}\left(\lapVar\fE-\fA_0+\delta\sum_{j=1}^{Q_{\fA}}\tilde \fL_j(t)+\varepsilon\fL(t)\right)
\]
is continuously differentiable with respect to $t$ and
\begin{equation} \label{eq:deriv}
\frac{d}{dt} F_\eps(\tilde \fL(t),\fL(t),\delta)
=
\delta\sum_{j=1}^{Q_{\fA}}\Re \,\bigl\langle  \fG(t),  \dot {\tilde \fL}_j(t)\bigr\rangle
+
\varepsilon \Re \bigl\langle  \fG(t),\dot \fL(t) \bigr\rangle,
\end{equation}
where $\fG(t)$ is the rank-one matrix
\begin{equation} \label{eq:freegrad}
\fG(t) =  \fu(t) \fv^*(t) \in \C^{\stateDim,\stateDim}.
\end{equation}
\end{lemma}

\begin{proof}
By \Cref{lem:eigderiv},  $F_\eps(\fL^{\cS}(t),\fL(t),\delta)$ is continuously differentiable, with
\begin{align*}
\begin{aligned}
\frac{ d }{dt} F_\eps(\tilde \fL(t),\fL(t),\delta)
\;=&\;
\dot \sigma_{\min}\left(\lapVar\fE-\fA_0+\delta\sum_{j=1}^{Q_{\fA}}\tilde \fL_j(t)+\varepsilon\fL(t)\right) \\
=&\;
\Re \left( \fu(t)^*\left(\delta\sum_{j=1}^{Q_{\fA}}\dot{\tilde \fL}_j(t)+ \varepsilon\dot{\fL}(t) \right)\fv(t)\right).
\end{aligned}
\end{align*}
Since
\[
\Re \left( \fu(t)^*\left(\delta\sum_{j=1}^{Q_{\fA}}\dot{\tilde \fL}_j(t)+ \varepsilon\dot{\fL}(t) \right) \fv(t) \right)
=
\delta\sum_{j=1}^{Q_{\fA}}\Re\, \bigl\langle \fu(t) \fv(t)^* ,\dot {\tilde \fL}_j(t)  \bigr\rangle
+
\varepsilon\Re\, \bigl\langle \fu(t) \fv(t)^* ,\dot \fL(t)  \bigr\rangle,
\]
we obtain \eqref{eq:deriv}--\eqref{eq:freegrad}.
\end{proof}

Our objective is now to use \eqref{eq:deriv} to minimize the functional $F_{\varepsilon}$. In the unconstrained setting, one would simply choose the descent directions equal to $-\fG(t)$. However, in our case, $\tilde \fL_j(t)$ and $\fL(t)$ must satisfy the constraints
\begin{equation}\label{eqn:con}
   \tilde \fL_j(t)\in\cS_{\fA_j},\; \|\tilde \fL_j(t)\|_{\Frob}=1,\; \text{for}\; j=1,\ldots,Q_{\fA};
   \quad\text{and}\quad
   \|\fL(t)\|_{\Frob}=1,
\end{equation}
which are generically not satisfied by $\fG(t)$. The next result provides the corresponding constrained gradient system.

\begin{lemma}(Constrained gradient flow)\label{lem:cons:grad:sys}
    Let $\tilde \fL=[\tilde \fL_1,\ldots,\tilde \fL_{Q_{\fA}}]\in\cS_{\fA}$, with $\tilde \fL_j\in\cS_{\fA_{j}}$ for $j=1,\ldots,Q_{\fA}$ where $\cS_{\fA_{j}}$ is defined in \eqref{eqn:manifold}. The solution $(\tilde \fL(t),\fL(t))$ of the optimization problem
    \begin{equation}\label{eqn:min:gra}
      \arg \min_{\fM(t)} \left(  \arg \min_{\tilde \fM(t)}\left( \delta\sum_{j=1}^{Q_{\fA}}\Re \,\bigl\langle  \fG(t),  \dot {\tilde \fM}_j(t)\bigr\rangle+\varepsilon \Re \bigl\langle  \fG(t),\dot \fM(t) \bigr\rangle\right)\right),
    \end{equation}
    with $\tilde \fL(t)$ and $\fL^{\cS}(t)$ satisfying \eqref{eqn:con} must satisfy the $Q_{\fA}+1$ matrix differential equations 
    \begin{subequations}\label{eqn:con:flow}
        \begin{align}
        \dot{\tilde\fL}_j(t) \;=&\; -\fPi^{\cS_{j}}\fG(t) + \Re \langle \,\fPi^{\cS_{j}}\fG(t), \tilde \fL_j(t) \,\rangle \tilde \fL_j(t),\quad j=1,\ldots,Q_{\fA}\,;\label{eqn:structured:flow}\\
            \dot \fL(t) \;=&\; -\fG(t) + \Re \langle\, \fG(t), \fL(t) \,\rangle \fL(t),\label{eqn:unstructured:flow}
        \end{align}
    \end{subequations}
    where $\fPi^{\cS_j}\fG(t)$ is the projection of $\fG(t)$ onto the set $\cS_{\fA_j}$ defined in \eqref{eqn:manifold}.
\end{lemma}
\begin{proof}
    First, let us observe that by deriving with respect to $t$ the condition $\|\fL(t)\|^2_{\Frob}=1$, we get
    \begin{equation}\label{eqn:costant:norm}
        0\;=\;\frac{d}{dt} \left(\|\fL(t)\|^2_{\Frob}\right)\;=\;\frac{d}{dt}\left( \langle\, \fL(t), \fL(t) \,\rangle\right)\;=\;2\Re \langle\, \dot\fL(t), \fL(t) \,\rangle, 
    \end{equation}
    thus implying that $\dot\fL(t)$ belongs to the set of matrices orthogonal to $\fL$. The same can be derived for $\tilde \fL_j(t)$ for $j=1,\ldots,Q_{\fA}$. Now, we observe that the minimum of problem \eqref{eqn:min:gra} is obtained by minimizing separately each term in the sum; therefore, we are left with solving
    \begin{subequations}
      \begin{align}
        \tilde \fL_j(t)\;\vcentcolon=& \;\arg\min_{\substack{\|\tilde \fM(t)\|_{\Frob}=1,\\\tilde \fM(t)\in\cS_{\fA_j}}}\Re \,\bigl\langle  \fG(t),  \dot{\tilde  \fM}(t)\bigr\rangle,\quad j=1,\ldots, Q_{\fA}\label{eqn:sub:eq:structured}\\
        \fL(t)\;\vcentcolon=& \;\arg\min_{\|\fM(t)\|_{\Frob}=1}\Re \,\bigl\langle  \fG(t),  \dot \fM(t)\bigr\rangle.\label{eqn:sub:eq:unstructured}
      \end{align}
    \end{subequations}
   We now focus on \eqref{eqn:sub:eq:unstructured}. The expression \eqref{eqn:unstructured:flow} is obtained directly from \eqref{eqn:sub:eq:unstructured} applying \cite[Lem.~2.3]{GL24}, which is based on the fact that the real part of the complex inner product on $\C^{\stateDim,\stateDim}$ coincides with the standard real inner product on $\R^{2\stateDim\times2\stateDim}$, or equivalently on $\R^{4\stateDim^2}$. Using the constant norm condition enforced by requiring $\dot \fM(t)$ to be orthogonal to $\fM(t)$, i.e., condition \eqref{eqn:costant:norm}, we see that \eqref{eqn:sub:eq:unstructured} is minimized by projecting the unconstrained steepest descent direction $-\fG(t)$ onto the tangent space of the manifold of matrices with a fixed Frobenius norm. This projection is given by
\[
\fPi \fG(t) \;=\; \fG(t) \;-\; \Re \langle \fG(t), \fL(t) \rangle\, \fL(t),
\]
which coincides with the right-hand side of \eqref{eqn:unstructured:flow}. Analogously, \eqref{eqn:structured:flow} follows directly from \eqref{eqn:sub:eq:structured} by repeating the same reasoning and noting that the constraint $\tilde \fL_j\in\cS_{\fA_j}$ additionally requires projecting $\fG(t)$ onto the subspace $\cS$.
\end{proof}

\begin{remark}[Extension to the case of $\fE$ being a parametric function]
   For $\cS_{\fE_j}$ and $\cS_{\fE_j}$ as defined in \eqref{eqn:manifold:E} and \eqref{eqn:manifold}, respectively, we now consider
    \begin{align*}
        \begin{aligned}
            \fL_{\fE}\;=&\;[\fL_{\fE_1},\ldots,\fL_{\fE_{Q_{\fE}}}]\in\cS_{\fE},\quad\fL_{\fE_j}\in\cS_{\fE_j}\\
            \fL_{\fA}\;=&\;[\fL_{\fA_1},\ldots,\fL_{\fA_{Q_{\fA}}}]\in\cS_{\fA},\quad\fL_{\fA_j}\in\cS_{\fA_j}.
        \end{aligned}
    \end{align*}
   Then, the functional to minimize becomes
    \begin{equation*}
         F_{\varepsilon}(\fL_{\fE},\fL_{\fA},\fL,\delta)
         =
         \sigma_{\min}\left(\lapVar\fE_0-\fA_0+\lapVar \delta\sum_{j=1}^{Q_{\fE}}\fL_{\fE_j}+\delta \sum_{j=1}^{Q_{\fA}}\fL_{\fA_j} + \eps \fL\right).
    \end{equation*}
   Following the same procedure as in \Cref{lem:gradient} and \ref{lem:cons:grad:sys}, this leads to a system of $Q_{\fE}+Q_{\fA}+1$ matrix \ODEs to be solved.
\end{remark}

\noindent
The following monotonicity property emerges naturally from the way gradient systems are constructed.

\begin{corollary}[Monotone decay of the functional]
\label{cor:monotone}
Let $\tilde \fL_j(t)$ and $\fL(t)$ satisfy the differential equations \eqref{eqn:con:flow}. Assume that
\[
\sigma_{\min}\left(\lapVar\fE-\fA_0+\delta\sum_{j=1}^{Q_{\fA}}\tilde \fL_j(t)+\eps\fL(t)\right)>0
\]
is a simple singular value that continuously depends on $t$. Consider $\tilde \fL=[\tilde 
\fL_1,\ldots,\tilde \fL_{Q_{\fA}}]$, then,
\begin{equation}
\frac{d}{dt} F_\eps (\tilde \fL(t),\fL(t),\delta)  \le  0.
\label{monotone}
\end{equation}
\end{corollary}

\begin{proof}
 Consider the inner product of \eqref{eqn:structured:flow} and \eqref{eqn:unstructured:flow} with $\dot {\tilde \fL}_j(t)$ and $\dot \fL(t)$, respectively. We get
 \begin{align}\label{eqn:positive:scalar:p}
 \begin{aligned}
       \|\dot{\tilde \fL}_j(t)\|^2_{\Frob}
       &=
       \Re \left\langle \dot {\tilde \fL}_j(t), -\fPi^{\cS_j}\fG(t) + \Re \langle \,\fPi^{\cS_j}\fG(t), \tilde \fL_j(t) \,\rangle\tilde  \fL_j(t) \right\rangle\\
       &=
       -\Re\langle \dot {\tilde \fL}_j(t) ,\fPi^{\cS_j}\fG(t)  \rangle,\qquad j=1,\ldots,Q_{\fA},\\[0.4em]
       \|\dot \fL(t)\|^2_{\Frob}
       &=
       \Re\left\langle \dot \fL(t), -\fG(t) + \Re \langle \,\fG(t), \fL(t) \,\rangle \fL(t) \right\rangle\\
       &=
       -\Re \langle \dot \fL(t) ,\fG(t)  \rangle,
    \end{aligned}
\end{align}
where we used the fact that both $\Re\langle \dot {\tilde \fL}_j(t) ,\tilde \fL_j(t)  \rangle$ and $\Re\langle \dot \fL(t) ,\fL(t)  \rangle$ vanish. Plugging \eqref{eqn:positive:scalar:p} into \eqref{eq:deriv}, we obtain
\begin{align}\label{eqn:negative:fun}
    \begin{aligned}
      \frac{d}{dt} F_\eps(\tilde \fL(t),\fL(t),\delta)
      &=
      \delta\sum_{j=1}^{Q_{\fA}}\Re \,\bigl\langle  \fG(t),  \dot{ \tilde{ \fL}}_j(t)\bigr\rangle+\varepsilon \Re \bigl\langle  \fG(t),\dot \fL(t) \bigr\rangle\\
      &=
      -\delta\sum_{j=1}^{Q_{\fA}}\|\dot{\tilde\fL}_j(t)\|^2_{\Frob}-\varepsilon\|\dot\fL(t)\|^2_{\Frob}\le0,
    \end{aligned}
\end{align}
where we used the fact that $\langle  \fG(t)-\fPi^{\cS_j}\fG(t),  \dot{\tilde \fL}_j(t)\rangle=0$ since $\dot {\tilde\fL}_j(t)\in\cS_{\fA_j}$ by \eqref{eqn:structured:flow}.
\end{proof}

\noindent
The stationary points of the differential equations \eqref{eqn:con:flow} are characterized as follows.

\begin{corollary}[Stationary points] 
\label{thm:stat}
Let $\fL^{\star}$ and $\tilde \fL^{\star}=[\tilde \fL^{\star}_{1},\ldots,\tilde \fL^{\star}_{Q_{\fA}}]\in\cS_{\fA}$ with $\tilde \fL^{\star}_{j}\in\cS_{\fA_j}$ for $j=1,\ldots,Q_{\fA}$, and $\| \fL^\star\|_{\Frob}=\| \tilde \fL^\star\|_{\Frob}=1$, be such that 
the singular value $\sigma_{\min}(\lapVar\fE-\fA_0+\delta\sum_{i=1}^{Q_{\fA}}\tilde \fL^{\star}_{j}+\varepsilon\fL^{\star})$ is simple and depends continuously on $\tilde \fL^{\star}$ and $\fL^{\star}$ in a neighborhood.
Let $\tilde \fL(t)=[\tilde \fL_1,\ldots,\tilde \fL_{Q_{\fA}}]\in \cS_{\fA}$ and $\fL(t)\in \C^{\stateDim,\stateDim}$ be, respectively, the solutions of \eqref{eqn:structured:flow} and \eqref{eqn:unstructured:flow} passing through $\tilde \fL^\star$ and $\fL^\star$. 
Then the following are equivalent:
\begin{enumerate}
    \item $\frac{ d }{dt} F_\eps \left(\tilde \fL(t), \fL(t),\delta \right)  = 0$.\label{item1}
    \item $\dot {\tilde \fL}_j(t) = 0$, for $j=1,\ldots,Q_{\fA}$, and $\dot \fL(t) = 0$.\label{item2}
    \item $\tilde{\fL}_j^\star$ is a real multiple of $\fPi^{\cS_j}\fG(t)$, for $j=1,\ldots,Q_{\fA}$, and $\fL^\star$ is a real multiple $\fG(t)$.\label{item3}
\end{enumerate}
\end{corollary}
\begin{proof} 
Clearly, if \cref{item3} holds, then \cref{item2} follows immediately from the right hand side of \eqref{eqn:con:flow} and the fact that $\tilde \fL^\star$ and $\fL^\star$ belong to the trajectory. \Cref{item2} implies \cref{item1} by \eqref{eqn:negative:fun} and vise versa. Finally, \cref{item1} implies \cref{item3} using \cref{item2} and \eqref{eqn:con:flow}.
\end{proof}

Every global minimum is, in particular, a local minimum, and by Corollary \ref{thm:stat} we can conclude that all local minima are stationary points of \eqref{eqn:con:flow}. Stationary points of the gradient flow that are not local minima are unstable. Consequently, one can generally expect a trajectory to converge to a local minimum. Moreover, Corollary \ref{thm:stat} shows that, under structured perturbations, the stationary points of the gradient system, and therefore the local minima of the functional, are precisely the projections onto $\cS_{\fA_j}$ of rank-one matrices, because $\fPi^{\cS_j}\fG(t)$ has this form; whereas for unstructured perturbations, the stationary points themselves are rank-one matrices; see \cref{item3}. This observation is crucial, as it motivates the search for a differential equation defined in the rank-1 matrix manifold that has the same stationary points but is computationally more efficient than evolving the dynamic in the space of complex square matrices of size $\stateDim$.

\subsubsection{Outer iteration: updating \texorpdfstring{$\tilde \delta$}{TEXT}}\label{subsec:outer:it}
For the solution of the scalar equation $\ophi(\delta)-\varepsilon=0$, we use the Newton method. We fix $\tilde \fL(\tilde \delta)$ and under the assumption that $\sigma_{\min}(\lapVar\fE-\fA_0+\delta\sum_{j=1}^{Q_{\fA}} \fL_j(\tilde \delta))$ is simple and greater than $0$, the derivative of $\ophi$ with respect to $\delta$ for the Newton iteration is given by the following formula:
\begin{equation}\label{eqn:der:phi}
    \ophi'(\delta)\;= \;\Re\left(\fu(\delta)^*\left(\sum_{j=1}^{Q_{\fA}}\tilde \fL_j(\tilde \delta)\right)\fv(\delta)\right)\;=\; \sum_{j=1}^{Q_{\fA}}\Re\left( \langle \fu(\delta)\fv(\delta)^* ,\tilde \fL_j(\tilde \delta) \rangle\right)
\end{equation}
where $\fu(\delta)$ and $\fv(\delta)$ are the left and right singular vectors associated with
$\sigma_{\min}(\lapVar\fE-\fA_0+\delta \sum_{j=1}^{Q_{\fA}}\tilde \fL_j(\tilde \delta))$. Therefore, recalling that $\tilde \delta$ is the current approximation of $\delta_{\varepsilon}$, its update during a single outer iteration step reads as
\begin{equation*}
    \delta\;=\;\tilde \delta-\frac{\sigma_{\min}\left(\lapVar\fE-\fA_0+\tilde \delta \sum_{j=1}^{Q_{\fA}}\tilde \fL_j(\tilde \delta)\right)-\varepsilon}{\sum_{j=1}^{Q_{\fA}}\Re\left(\fu(\tilde \delta)^*\tilde \fL_j(\tilde \delta)\fv(\tilde \delta)\right)},
\end{equation*}
assuming that $\sum_{j=1}^{Q_{\fA}}\Re\left(\fu(\tilde \delta)^*\tilde \fL_j(\tilde \delta)\fv(\tilde \delta)\right)\neq0$. Note that, due to Corollary \ref{cor:monotone}, we also have
\[\Re\left(\fu(\tilde \delta)^*\tilde \fL_j(\tilde \delta)\fv(\tilde \delta)\right)\le0,\quad\text{for}\quad j=1,\ldots,Q_{\fA}.\]
If the assumption of a simple smallest singular value does not hold, one can always resort to the bisection method to update $\tilde \delta$.
\begin{remark}[Extension to the case of $\fE$ being a parametric function]
    For this case we have
    \begin{equation}\label{eqn:sv}
    \sigma_{\min}\left(\lapVar\fE_0-\fA_0+\delta\lapVar\sum_{j=1}^{Q_{\fE}}\fL_{\fE_j}(\tilde \delta)+\delta\sum_{j=1}^{Q_{\fA}}\fL_{\fA_j}(\tilde \delta)\right),
    \end{equation}
    thus the application of Newton or bisection method follows straightforwardly by simply observing that, in the case of Newton method, \eqref{eqn:der:phi} is replaced by
\begin{equation*}
    \ophi'(\delta)\;= \;\Re\left(\fu(\delta)^*\left(\lapVar\sum_{j=1}^{Q_{\fE}}\fL_{\fE_j}(\tilde \delta)+\sum_{j=1}^{Q_{\fA}}\fL_{\fA_j}(\tilde \delta)\right)\fv(\delta)\right),
\end{equation*}
being $\fu(\delta)$, $\fv(\delta)$ the left and right singular vectors associated with \eqref{eqn:sv}.
\end{remark}

\subsection{Algorithm and computational aspects}\label{sec3:comp:aspects}

By combining the results from \Cref{subsec:r1-ode-eps-rad} and \Cref{subsec:outer:it}, we obtain \Cref{alg1}, which provides an upper bound of $\delta_{\varepsilon}$, that is, the $\cS_{\fA}$-structured $\varepsilon$-distance to singularity of $\lapVar\fE-\fA_0$. 

\begin{remark}
We note that the output of \Cref{alg1} is only guarantee to be an upper bound on $\delta_{\varepsilon}$. This is because the inner iterations, which rely on solving a gradient-based system, can in general ensure convergence only to local minima. Moreover, the convergence of the outer iterations, which are based on Newton or bisection methods, is also influenced by the choice of the initial point. Despite this, if one carefully applies the step-size control, the local convergence results are still useful in our framework for identifying a relevant set of parameters. Subsequently, for validation purposes, a posteriori error control can be employed to detect parameters that may have been incorrectly included.
\end{remark}

\begin{algorithm}[t]
\caption{$\cS_{\fA}$-structured $\varepsilon$-distance from singularity of $\lapVar\fE-\fA_0$}\label{alg1}
\label{alg:strcatured:eps:distance:singularity}
\begin{algorithmic}[1] 
\Require The matrices $\fE$, $\fA_0$, the desired unstructured distance from singularity $\varepsilon>0$, the structured set $\cS_j$ (see \eqref{eqn:manifold}) for $j=1,\ldots,Q_{\fA}$, the complex point of interest $\lapVar$, exit tolerance $\tol$, maximum iterations $k_{\max}$
\Ensure Upper bound for the $\cS_{\fA}$-structured $\varepsilon$-distance from singularity $\delta_{\varepsilon}$, the structured perturbation matrices $\tilde \fL_j(\delta_\varepsilon)$ for $j=1,\ldots,Q_{\fA}$

\State Set $\delta_0,\delta_1=0$, $\phi(\delta_0)=\infty$, $\phi(\delta_1)=\sigma_{\min}(\lapVar\fE-\fA_0)$ with $\fu_1$ and $\fv_1$ associated left and right singular vectors. Set $\fL(\delta_0)=-\fu_1\fv_1^{*}$, $\tilde \fL_j(\delta_0)=-\fPi^{\cS_j}\left(\fu_1\fv_1^{*}\right)$ for $j=1,\ldots,Q_{\fA}$
\For{$k = 1$ to $k_{\max}$}

 \State   Compute $\tilde \fL_j(\delta_k)$ for $j=1,\ldots,Q_{\fA}$, $\fL(\delta_k)$, $\phi(\delta_k)$ by integrating the constrained flow systems \eqref{eqn:con:flow} with initial datum $\tilde \fL_j(\delta_{k-1})$ and $\tilde \fL(\delta_{k-1})$. (This is the \textbf{inner iteration})

\If{$|\phi(\delta_k)-\phi(\delta_{k-1})|\le\tol$}
        \State go to line 10
    \EndIf
 
 \State Compute $\fu(\delta_k)$ and $\fv(\delta_k)$, left and right singular vectors associated to $\sigma_{\min}(\lapVar\fE-\fA_0+\delta_k \sum_{j=1}^{Q_{\fA}}\tilde \fL_j(\delta_k))$ 
 \State Compute \[\delta_{k+1}=\delta_k-\left(\sigma_{\min}\left(\lapVar\fE-\fA_0+\delta_k \sum_{j=1}^{Q_{\fA}}\tilde \fL_j(\delta_k)\right)-\varepsilon\right)\left(\sum_{j=1}^{Q_{\fA}}\Re\left(\fu(\delta_k)^*\tilde \fL_j(\delta_k)\fv(\delta_k)\right)\right)^{-1}\]

\EndFor
\State \Return $\delta_{\varepsilon}=\delta_{k}$, $\tilde \fL(\delta_{\eps})=[\tilde \fL_1(\delta_{\eps}),\ldots,\tilde \fL_{Q_{\fA}}(\delta_{\eps})]$
\end{algorithmic}
\end{algorithm}

The use of a standard Euler integrator for the numerical approximation of the stationary points of \eqref{eqn:con:flow}, equipped with a step-size control based on the monotonicity of the functional, usually provides good results. The use of more sophisticated integrators equipped with trust-region techniques (such as Armijo's rule) is discussed in \cite{GL24}. From a computational perspective, formulating the problem as a gradient system provides two key benefits. First, the unstructured perturbation matrix $\fL(t)$ has rank one, so its approximation can be computed by integrating its rank one factors; see, for instance, \cite[Lem.~3.4]{book}, instead of forming $\fL(t)$ explicitly; thus, one works with $\stateDim$-dimensional vectors rather than $\stateDim\times\stateDim$ full matrices. Second, the $Q_{\fA}$ structured perturbation matrices $\tilde \fL_j(t)$ are computed as the projection of a rank-one matrix onto the prescribed structure. When the structure is defined by a sparsity pattern, as in the problems we consider, this makes it possible to work only with sparse matrices. Consequently, integrating \eqref{eqn:structured:flow} requires neither storing full matrices nor performing $\mathcal O(\stateDim^2)$ floating-point operations.

We also note that an additional acceleration of the numerical integration of \eqref{eqn:con:flow} can be obtained by replacing the $\stateDim$-dimensional problem with a reduced one of dimension $\stateDimRed\ll\stateDim$, constructed via projection onto appropriately chosen subspaces; see, for example, \cite[Sec.~5]{ManMG26} for the case of the smallest singular value. This can be crucial for speeding up \Cref{alg1}, since time integration with step-size control typically requires solving a large number of singular value problems, which can become expensive, even with sparse matrices, as $\stateDim$ grows.

We conclude by discussing the various eigenvalue problem solutions that are required for the execution of \Cref{alg1}. Indeed, the step-size control for the \ODE numerical integration, as well as the evaluation of the derivative of the functional in the Newton method for the outer iterations, involves solving several spectral problems associated with the smallest singular value. Since these problems are expressed as sums of rank-one and sparse matrices, one can exploit iterative methods that require only matrix-vector products. In this way, assuming convergence of the iterative method is achieved within a number of iterations much smaller than $\stateDim$, the computational cost scales linearly with $\stateDim$.


\section{Numerical experiments} \label{sec:num}
In this section, we provide numerical experiments that substantiate both the theoretical findings and the proposed methodology. We first verify the \CIM for \DAE approach outlined in \Cref{sec:constr} using two benchmark \DAE examples: a constrained mass-spring-damper system and the Stokes problem. We then demonstrate the framework introduced in \Cref{sec:str-uns} through numerical results obtained from discretized parametric \PDEs.

All calculations were performed with \matlab~2024b on a MacBook Pro with an Apple M2 Pro processor and 16GB of RAM. 

\vspace{0.2cm}
\noindent\fbox{%
	\parbox{0.98\textwidth}{%
		The code and data used to generate the subsequent results are accessible via
		\begin{center}
			\url{https://doi.org/10.5281/zenodo.21264292}
		\end{center}
		under MIT Common License.
	}%
}\\[.2em]

\subsection{A \CIM for \DAE: examples of applications}  \label{sec:num:CIM:DAE}
The routine for determining the profile $\Gamma$ for the \CIM approximation is derived from \cite{GugLM21}, where the standard \ODE case was treated. In essence, given target accuracy $\tol$ the method constructs a contour to approximate the solution with that accuracy and it also provides an estimate of the number of quadrature points necessary to reach the given precision.
\subsubsection{A \CIM for the constrained mass-spring-damper system} \label{sec:num:mass:spring}

We consider the holonomically constrained mass-spring-damper system presented in \cite[Sec.~4]{MehS05}. The vibration of this system is described by the descriptor system
\begin{align}\label{eq34}
	\left\{\quad
	\begin{aligned}
		\dot{\fp}(t) &= \fv(t),\\ 
		\fM\dot{\fv}(t) &=\fK\fp(t)+\fD\fv(t)-\fG^\T{\lambda}(t)+\fB_2\inp(t),\\
		\zeroVec &= \fG\fp(t),
	\end{aligned}\right.
\end{align}%
where $\fp(t) \in \R^g$ is the position vector, $\fv(t) \in \R^g$ is the velocity vector, $\flambda(t)\in \R$ is the Lagrange multiplier, $\fM = \text{diag}(m_1,\ldots,m_g)$ is the mass matrix, $\fD$ and $\fK$ are the tridiagonal damping and stiffness matrices, $\fG = [1, 0, \ldots,0, -1]\in \R^g$ is the constraint matrix, $\fB_2 = \fe_1$, where $\fe_j$ denotes the $j$th column of the identity matrix $\fI_{g}$. The descriptor system arising from \eqref{eq34} is of index $\indDAE=3$ and its associated matrices are
\begin{align}
	\label{eq:ref:con:mas:spr}
	\fE &\vcentcolon= \begin{bmatrix}
		\fI_g & \zeroMat & \zeroMat \\
		\zeroMat & \fM & \zeroMat \\
		\zeroMat & \zeroMat & \zeroMat
	\end{bmatrix}, &
	\fA &\vcentcolon= \begin{bmatrix}
		\zeroMat & \fI_g & \zeroMat \\
		\fK& \fD & -\fG^\T\\
		\fG & \zeroMat & \zeroMat \\
	\end{bmatrix}, &
	\fB &\vcentcolon= \begin{bmatrix}
		\zeroVec \\
		\fB_2\\
		\zeroVec \\
	\end{bmatrix} .
\end{align}
Considering $g$ masses, the state is given by $\stx(t)\vcentcolon=[\fp(t)^{\T},\fv(t)^{\T},\lambda(t)]^{\T}$ and thus the system dimension is $\stateDim=2g+1$, while input and output dimensions are $\inpDim=1$ and $\outDim=3$. The specific setting of the parameters is taken from \cite[Sec.~4]{MehS05}. Note that the matrices in \eqref{eq:ref:con:mas:spr} are sparse and the kernel of $\fE$ is of dimension one. 
\begin{figure}[t]
	\centering{
	\subfigure[$T=100$ and $\inp(t)=t^2$]{\label{subfig:1a}
%
\tikzexternaldisable
\begin{tikzpicture}

\begin{axis}[%
	width=0.64*\imageWidth,
	height=\imageHeight,
	scale only axis,
	scaled ticks=false,
	grid=both,
	grid style={line width=.1pt, draw=gray!10},
	major grid style={line width=.2pt,draw=gray!50},
	axis lines*=left,
	axis line style={line width=\lineWidth},
xmin=1,
xmax=45,
xlabel style={font=\color{white!15!black}},
xlabel={N},
ymode=log,
ymin=1e-12,
ymax=1e5,
yminorticks=true,
ylabel style={font=\color{white!15!black}},
ylabel={$\|\stx(T)-\stx_N{T}\|$},
	axis background/.style={fill=white},
	legend style={%
		legend cell align=left, 
		align=left, 
		font=\tiny,
		draw=white!15!black,
		at={(1.0,1.0)},
		anchor=north east,},
]


\addplot [color=mycolor1, line width=\lineWidth, forget plot]
table [x index=0, y index=1, col sep=comma]{img/DataCSV/Fig1a_tol1.000000e-10.csv};

\addplot [color=mycolor1, dotted, line width=\lineWidth]
table [x index=0, y index=2, col sep=comma]{img/DataCSV/Fig1a_tol1.000000e-10.csv};

\addlegendentry{$\tol=10^{-10}$}

\pgfplotstableread[col sep=comma]{img/DataCSV/Fig1a_QP.csv}\datatable

\pgfplotstablegetelem{3}{x1}\of{\datatable}
\let\xcoord\pgfplotsretval

\pgfplotstablegetelem{3}{y1}\of{\datatable}
\let\ycoord\pgfplotsretval

\addplot [color=mycolor1, line width=\lineWidth, only marks, mark=o, mark options={solid, mycolor1}, forget plot]
 coordinates {(\xcoord,\ycoord)};


\addplot [color=mycolor3, line width=\lineWidth, forget plot]
table [x index=0, y index=1, col sep=comma]{img/DataCSV/Fig1a_tol1.000000e-07.csv};

\addplot [color=mycolor3, densely dotted, line width=\lineWidth]
table [x index=0, y index=2, col sep=comma]{img/DataCSV/Fig1a_tol1.000000e-07.csv};

\addlegendentry{$\tol=10^{-7}$}

\pgfplotstableread[col sep=comma]{img/DataCSV/Fig1a_QP.csv}\datatable

\pgfplotstablegetelem{2}{x1}\of{\datatable}
\let\xcoord\pgfplotsretval

\pgfplotstablegetelem{2}{y1}\of{\datatable}
\let\ycoord\pgfplotsretval

\addplot [color=mycolor3, line width=\lineWidth, only marks, mark=o, mark options={solid, mycolor3}, forget plot]
  coordinates {(\xcoord,\ycoord)};


\addplot [color=mycolor4, line width=\lineWidth, forget plot]
table [x index=0, y index=1, col sep=comma]{img/DataCSV/Fig1a_tol1.000000e-04.csv};

\addplot [color=mycolor4, loosely dashdotted, line width=\lineWidth]
table [x index=0, y index=2, col sep=comma]{img/DataCSV/Fig1a_tol1.000000e-04.csv};

\addlegendentry{$\tol=10^{-4}$}

\pgfplotstableread[col sep=comma]{img/DataCSV/Fig1a_QP.csv}\datatable

\pgfplotstablegetelem{1}{x1}\of{\datatable}
\let\xcoord\pgfplotsretval

\pgfplotstablegetelem{1}{y1}\of{\datatable}
\let\ycoord\pgfplotsretval

\addplot [color=mycolor4, line width=\lineWidth, only marks, mark=o, mark options={solid, mycolor4}, forget plot]
  coordinates {(\xcoord,\ycoord)};


\addplot [color=mycolor2, line width=\lineWidth, forget plot]
table [x index=0, y index=1, col sep=comma]{img/DataCSV/Fig1a_tol1.000000e-01.csv};

\addplot [color=mycolor2, densely dashdotted, line width=\lineWidth]
table [x index=0, y index=2, col sep=comma]{img/DataCSV/Fig1a_tol1.000000e-01.csv};

\addlegendentry{$\tol=10^{-1}$}

\pgfplotstableread[col sep=comma]{img/DataCSV/Fig1a_QP.csv}\datatable

\pgfplotstablegetelem{0}{x1}\of{\datatable}
\let\xcoord\pgfplotsretval

\pgfplotstablegetelem{0}{y1}\of{\datatable}
\let\ycoord\pgfplotsretval

\addplot [color=mycolor2, line width=\lineWidth, only marks, mark=o, mark options={solid, mycolor2}, forget plot]
  coordinates {(\xcoord,\ycoord)};

\end{axis}
\end{tikzpicture}%
}
	\subfigure[$T=100$ and $\inp(t)=10\ee^{-t}+1$]{\label{subfig:1b}
%
\tikzexternaldisable
\begin{tikzpicture}

\begin{axis}[%
width=0.64*\imageWidth,
	height=\imageHeight,
	scale only axis,
	scaled ticks=false,
	grid=both,
	grid style={line width=.1pt, draw=gray!10},
	major grid style={line width=.2pt,draw=gray!50},
	axis lines*=left,
	axis line style={line width=\lineWidth},
xmin=1,
xmax=45,
xlabel style={font=\color{white!15!black}},
xlabel={N},
ymode=log,
ymin=1e-12,
ymax=1e3,
yminorticks=true,
ylabel style={font=\color{white!15!black}},
	axis background/.style={fill=white},
	legend style={%
		legend cell align=left, 
		align=left, 
		font=\tiny,
		draw=white!15!black,
		at={(1.0,1.0)},
		anchor=north east,},
]


\addplot [color=mycolor1, line width=\lineWidth, forget plot]
table [x index=0, y index=1, col sep=comma]{img/DataCSV/Fig1b_tol1.000000e-10.csv};

\addplot [color=mycolor1, dotted, line width=\lineWidth]
table [x index=0, y index=2, col sep=comma]{img/DataCSV/Fig1b_tol1.000000e-10.csv};


\pgfplotstableread[col sep=comma]{img/DataCSV/Fig1b_QP.csv}\datatable

\pgfplotstablegetelem{3}{x1}\of{\datatable}
\let\xcoord\pgfplotsretval

\pgfplotstablegetelem{3}{y1}\of{\datatable}
\let\ycoord\pgfplotsretval

\addplot [color=mycolor1, line width=\lineWidth, only marks, mark=o, mark options={solid, mycolor1}, forget plot]
 coordinates {(\xcoord,\ycoord)};


\addplot [color=mycolor3, line width=\lineWidth, forget plot]
table [x index=0, y index=1, col sep=comma]{img/DataCSV/Fig1b_tol1.000000e-07.csv};

\addplot [color=mycolor3, densely dotted, line width=\lineWidth]
table [x index=0, y index=2, col sep=comma]{img/DataCSV/Fig1b_tol1.000000e-07.csv};


\pgfplotstableread[col sep=comma]{img/DataCSV/Fig1b_QP.csv}\datatable

\pgfplotstablegetelem{2}{x1}\of{\datatable}
\let\xcoord\pgfplotsretval

\pgfplotstablegetelem{2}{y1}\of{\datatable}
\let\ycoord\pgfplotsretval

\addplot [color=mycolor3, line width=\lineWidth, only marks, mark=o, mark options={solid, mycolor3}, forget plot]
  coordinates {(\xcoord,\ycoord)};


\addplot [color=mycolor4, line width=\lineWidth, forget plot]
table [x index=0, y index=1, col sep=comma]{img/DataCSV/Fig1b_tol1.000000e-04.csv};

\addplot [color=mycolor4, loosely dashdotted, line width=\lineWidth]
table [x index=0, y index=2, col sep=comma]{img/DataCSV/Fig1b_tol1.000000e-04.csv};


\pgfplotstableread[col sep=comma]{img/DataCSV/Fig1b_QP.csv}\datatable

\pgfplotstablegetelem{1}{x1}\of{\datatable}
\let\xcoord\pgfplotsretval

\pgfplotstablegetelem{1}{y1}\of{\datatable}
\let\ycoord\pgfplotsretval

\addplot [color=mycolor4, line width=\lineWidth, only marks, mark=o, mark options={solid, mycolor4}, forget plot]
  coordinates {(\xcoord,\ycoord)};


\addplot [color=mycolor2, line width=\lineWidth, forget plot]
table [x index=0, y index=1, col sep=comma]{img/DataCSV/Fig1b_tol1.000000e-01.csv};

\addplot [color=mycolor2, densely dashdotted, line width=\lineWidth]
table [x index=0, y index=2, col sep=comma]{img/DataCSV/Fig1b_tol1.000000e-01.csv};


\pgfplotstableread[col sep=comma]{img/DataCSV/Fig1b_QP.csv}\datatable

\pgfplotstablegetelem{0}{x1}\of{\datatable}
\let\xcoord\pgfplotsretval

\pgfplotstablegetelem{0}{y1}\of{\datatable}
\let\ycoord\pgfplotsretval

\addplot [color=mycolor2, line width=\lineWidth, only marks, mark=o, mark options={solid, mycolor2}, forget plot]
  coordinates {(\xcoord,\ycoord)};
\end{axis}
\end{tikzpicture}%
}
 	\subfigure[$T=10$ and $\tol=10^{-8}$]{\label{subfig:1c}
		\tikzexternaldisable
\begin{tikzpicture}

\begin{axis}[%
	width=0.64*\imageWidth,
	height=\imageHeight,
scale only axis,
	scaled ticks=false,
	grid=both,
	grid style={line width=.1pt, draw=gray!10},
	major grid style={line width=.2pt,draw=gray!50},
	axis lines*=left,
	axis line style={line width=\lineWidth},
xmode=log,
xmin=50,
xmax=2e5,
xlabel={$\stateDim$},
xminorticks=true,
ymode=log,
ymin=0.005,
ymax=1e5,
yminorticks=true,
axis background/.style={fill=white},
	legend style={%
		legend cell align=left, 
		align=left, 
		font=\scriptsize,
		draw=white!15!black,
		at={(0.71,1)},
		anchor=north east,},
]
\addplot [color=mycolor1, line width=\lineWidth, mark=o, mark options={solid, mycolor1}]
table [x=x1, y=y1, col sep=comma]{img/DataCSV/Fig1c.csv};
\addlegendentry{$\texttt{ODE15s}$}

\addplot [color=mycolor2, line width=\lineWidth, mark=square, mark options={solid, mycolor2}]
table [x=x1, y=y5, col sep=comma]{img/DataCSV/Fig1c.csv};
\addlegendentry{Time \QR}

\addplot [color=mycolor1, dashed, line width=\lineWidth, mark=o, mark options={solid, mycolor1}]
table [x=x1, y=y2, col sep=comma]{img/DataCSV/Fig1c.csv};
\addlegendentry{Time dec.}

\addplot [color=mycolor2, dashed, line width=\lineWidth, mark=square, mark options={solid, mycolor2}]
table [x=x1, y=y4, col sep=comma]{img/DataCSV/Fig1c.csv};
\addlegendentry{Time $\Gamma$}

\addplot [color=mycolor4, dashed, line width=\lineWidth]
table [x=x1, y=y7, col sep=comma]{img/DataCSV/Fig1c.csv};
\addlegendentry{$\mathcal{O}(\stateDim)$}

\end{axis}
\end{tikzpicture}%
}
 }
	\caption{Constrained mass-spring-damper system. Decay of the quadrature error for different values of the target accuracy $\tol$ (left and center). Computational time with respect to dimension of the problem (right).}
	\label{fig1}%
\end{figure}
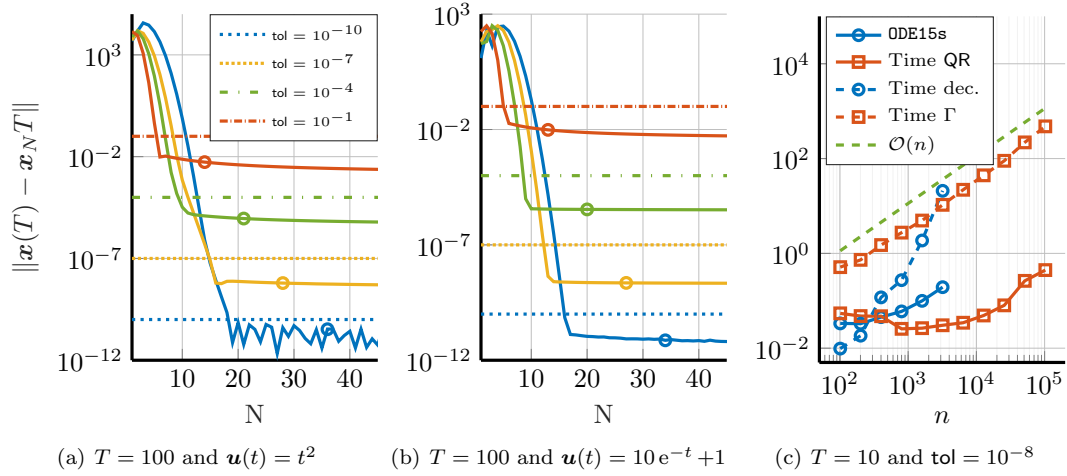%

We set $T=100$ and $g=5000$; then we consider the approximation of $\stx(T)$ for different values of the target precision $\tol$. The reference solution is obtained in two steps: we first decouple the differential and impulsive components by computing the decoupling matrices $\fS$ and $\fT$; next, we integrate the resulting differential system \eqref{eqn:DAE:slow} with high accuracy by employing the built-in \matlab~routine \texttt{ODE15s}, and then evaluate the impulsive part \eqref{eqn:sol:sub:fast} at time $T$. Finally, we reconstruct $\stx(T)$ via the relation $\stx(T)=\fT(\stx^{\diff}(T)\oplus\stx^{\imp}(T))$.

\begin{remark}
    Instead of separating the differential and impulsive components—which would require the computation of $\fS$ and $\fT$, one may alternatively differentiate the constraint equation $\indDAE$ multiple times until a \ODE system is obtained. This operation can be carried out analytically and, in the case of mass–spring–damper systems, also efficiently, since it still results in the integration of a \ODE system with sparse matrices. However, the subsequent time integration by time-stepping becomes more delicate and may even turn unstable; see \cite{BreCP95}, which is consistent with our own numerical observations. Therefore, when computing reference solutions, we always employ the decoupling strategy.
\end{remark}
The results are shown in \Cref{subfig:1a} and \Cref{subfig:1b}, corresponding to the input functions $\inp(t)=t^2$ and $\inp(t)=10\ee^{-t}+1$, respectively. The error consistently remains below the prescribed tolerance $\tol$, and the estimated number of quadrature points required to achieve this integration accuracy is indicated by a circle on the corresponding error curve. We observe that the Laplace transforms of these two inputs are given by $\hat{\inp}(\lapVar)=2/z^3$ and $\hat{\inp}(\lapVar)=10/(z+1)+1/z$, respectively. Since $\indDAE=3$, \Cref{lemma2} requires that the source term exhibits asymptotic decay in the Laplace transform of its first two derivatives, a condition that is satisfied by both input functions under consideration. 

Finally, we compare the computation times, as a function of the dimension of the problem $\stateDim$, required to approximate $\stx(T)$ at the final time $T=10$ with precision $\tol=10^{-8}$. The comparison is conducted between the decoupling approach, for which we report the computational time as \emph{Time dec.}, combined with time integration via \texttt{ODE15s}, and the \CIM applied to the \DAE. With this objective, we decompose the overall computational time of the \CIM into two distinct contributions: the time required to construct the integration profile $\Gamma$, denoted as Time $\Gamma$, and the time associated with the evaluation of the quadrature rule \eqref{eqn:quad:app}, denoted as Time \QR. The latter is predominantly determined by the solution of the $N/2$ linear systems arising from the quadrature formulation. We emphasize that these linear systems are mutually independent and can therefore be solved in parallel. From \Cref{subfig:1c} we observe that, for the \CIM applied to \DAE, the overall computational cost is mainly dominated by the construction of the integration profile $\Gamma$, while the solution of the linear systems involved in the quadrature rule \eqref{eqn:quad:app} is several orders of magnitude faster. The decoupling procedure is initially less expensive than the construction of $\Gamma$; however, as the dimension of the problem increases, its cost increases and eventually becomes prohibitive, which is why we do not report results for larger values of $\stateDim$. A comparison between the running time of \texttt{ODE15s} and that of the quadrature rule further shows that solving a small number of sparse linear systems is more efficient than time-stepping–based integration, since the latter requires computing the full trajectory over the entire time interval $[0,T]$.

It is important to emphasize that neither the decoupling routine nor the construction of $\Gamma$ has been optimized in the present work. The algorithm used to construct $\Gamma$ has substantial potential for improvement and can generally be tailored to the specific class of problems under consideration. For example, in the constrained mass-spring-damper system, the integration profile remained unchanged at different values of $\stateDim$. This observation suggests that a practical strategy may consist of designing the integration contour for relatively small values of $\stateDim$ and subsequently reusing it for substantially larger dimensions. The decoupling phase can also be accelerated when sparse representations of $\fS$ and $\fT$ are available, thereby avoiding the explicit formation of the decoupled system. The main point illustrated by this plot is that \CIM for \DAE provides a numerically stable tool (if $\Gamma$ is carefully constructed) to approximate the solution of a \DAE at a prescribed time (or over a suitable time window) without the need to decouple the \DAE system or to engage in index reduction techniques and carefully tuned, stable time-stepping integrators. The main drawback is the computational effort required to determine an appropriate placement of the contour $\Gamma$ and the fact that the full time trajectory is not directly obtained. Nevertheless, this approach can be more flexible and convenient in situations where decoupling or index reduction is prohibitively expensive and only a portion of the trajectory in time is of interest.

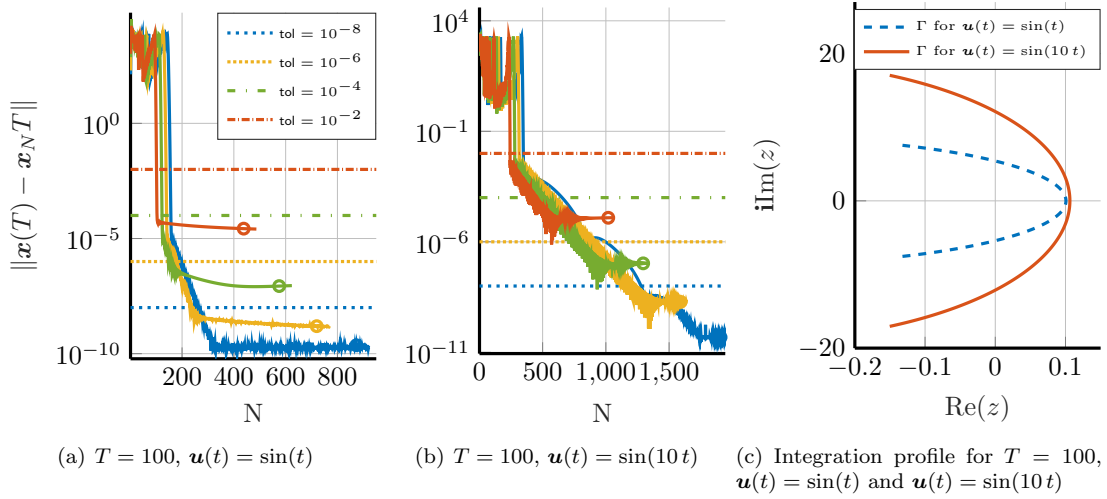
\begin{figure}[t]
	\centering{
	\subfigure[$T=100$, $\inp(t)=\sin(t)$]{
%
\tikzexternaldisable
\begin{tikzpicture}

\begin{axis}[%
	width=0.64*\imageWidth,
	height=\imageHeight,
	scale only axis,
	scaled ticks=false,
	grid=both,
	grid style={line width=.1pt, draw=gray!10},
	major grid style={line width=.2pt,draw=gray!50},
	axis lines*=left,
	axis line style={line width=\lineWidth},
xmin=1,
xmax=950,
xlabel style={font=\color{white!15!black}},
xlabel={N},
ymode=log,
ymin=6e-11,
ymax=63860.480072029,
yminorticks=true,
ylabel style={font=\color{white!15!black}},
ylabel={$\|\stx(T)-\stx_N{T}\|$},
	axis background/.style={fill=white},
	legend style={%
		legend cell align=left, 
		align=left, 
		font=\tiny,
		draw=white!15!black,
		at={(1.00,1.00)},
		anchor=north east,},
]


\addplot [color=mycolor1, line width=\lineWidth, forget plot]
table [x index=0, y index=1, col sep=comma]{img/DataCSV/Fig1a_bis_tol1.000000e-08.csv};

\addplot [color=mycolor1, dotted, line width=\lineWidth]
table [row sep=crcr]{
1 1e-8 \\
950 1e-8 \\
};

\addlegendentry{$\tol=10^{-8}$}

\pgfplotstableread[col sep=comma]{img/DataCSV/Fig1a_bis_QP.csv}\datatable

\pgfplotstablegetelem{3}{x1}\of{\datatable}
\let\xcoord\pgfplotsretval

\pgfplotstablegetelem{3}{y1}\of{\datatable}
\let\ycoord\pgfplotsretval

\addplot [color=mycolor1, line width=\lineWidth, only marks, mark=o, mark options={solid, mycolor1}, forget plot]
 coordinates {(\xcoord,\ycoord)};


\addplot [color=mycolor3, line width=\lineWidth, forget plot]
table [x index=0, y index=1, col sep=comma]{img/DataCSV/Fig1a_bis_tol1.000000e-06.csv};

\addplot [color=mycolor3, densely dotted, line width=\lineWidth]
table [row sep=crcr]{
1 1e-6 \\
950 1e-6 \\
};

\addlegendentry{$\tol=10^{-6}$}

\pgfplotstableread[col sep=comma]{img/DataCSV/Fig1a_bis_QP.csv}\datatable

\pgfplotstablegetelem{2}{x1}\of{\datatable}
\let\xcoord\pgfplotsretval

\pgfplotstablegetelem{2}{y1}\of{\datatable}
\let\ycoord\pgfplotsretval

\addplot [color=mycolor3, line width=\lineWidth, only marks, mark=o, mark options={solid, mycolor3}, forget plot]
  coordinates {(\xcoord,\ycoord)};


\addplot [color=mycolor4, line width=\lineWidth, forget plot]
table [x index=0, y index=1, col sep=comma]{img/DataCSV/Fig1a_bis_tol1.000000e-04.csv};

\addplot [color=mycolor4, loosely dashdotted, line width=\lineWidth]
table [row sep=crcr]{
1 1e-4 \\
950 1e-4 \\
};

\addlegendentry{$\tol=10^{-4}$}

\pgfplotstableread[col sep=comma]{img/DataCSV/Fig1a_bis_QP.csv}\datatable

\pgfplotstablegetelem{1}{x1}\of{\datatable}
\let\xcoord\pgfplotsretval

\pgfplotstablegetelem{1}{y1}\of{\datatable}
\let\ycoord\pgfplotsretval

\addplot [color=mycolor4, line width=\lineWidth, only marks, mark=o, mark options={solid, mycolor4}, forget plot]
  coordinates {(\xcoord,\ycoord)};


\addplot [color=mycolor2, line width=\lineWidth, forget plot]
table [x index=0, y index=1, col sep=comma]{img/DataCSV/Fig1a_bis_tol1.000000e-02.csv};

\addplot [color=mycolor2, densely dashdotted, line width=\lineWidth]
table [row sep=crcr]{
1 1e-2 \\
950 1e-2 \\
};

\addlegendentry{$\tol=10^{-2}$}

\pgfplotstableread[col sep=comma]{img/DataCSV/Fig1a_bis_QP.csv}\datatable

\pgfplotstablegetelem{0}{x1}\of{\datatable}
\let\xcoord\pgfplotsretval

\pgfplotstablegetelem{0}{y1}\of{\datatable}
\let\ycoord\pgfplotsretval

\addplot [color=mycolor2, line width=\lineWidth, only marks, mark=o, mark options={solid, mycolor2}, forget plot]
  coordinates {(\xcoord,\ycoord)};

\end{axis}
\end{tikzpicture}%
\label{subfig:1bis:a}}
	\subfigure[$T=100$, $\inp(t)=\sin(10\,t)$]{
%
\tikzexternaldisable
\begin{tikzpicture}

\begin{axis}[%
width=0.64*\imageWidth,
	height=\imageHeight,
	scale only axis,
	scaled ticks=false,
	grid=both,
	grid style={line width=.1pt, draw=gray!10},
	major grid style={line width=.2pt,draw=gray!50},
	axis lines*=left,
	axis line style={line width=\lineWidth},
xmin=1,
xmax=1944,
xlabel style={font=\color{white!15!black}},
xlabel={N},
ymode=log,
ymin=8.75422047243395e-12,
ymax=4e4,
yminorticks=true,
ylabel style={font=\color{white!15!black}},
	axis background/.style={fill=white},
	legend style={%
		legend cell align=left, 
		align=left, 
		font=\tiny,
		draw=white!15!black,
		at={(1.00,1.00)},
		anchor=north east,},
]


\addplot [color=mycolor1, line width=\lineWidth, forget plot]
table [x index=0, y index=1, col sep=comma]{img/DataCSV/Fig1b_bis_tol1.000000e-08.csv};

\addplot [color=mycolor1, dotted, line width=\lineWidth]
table [row sep=crcr]{
1 1e-8 \\
1944 1e-8 \\
};


\pgfplotstableread[col sep=comma]{img/DataCSV/Fig1b_bis_QP.csv}\datatable

\pgfplotstablegetelem{3}{x1}\of{\datatable}
\let\xcoord\pgfplotsretval

\pgfplotstablegetelem{3}{y1}\of{\datatable}
\let\ycoord\pgfplotsretval

\addplot [color=mycolor1, line width=\lineWidth, only marks, mark=o, mark options={solid, mycolor1}, forget plot]
 coordinates {(\xcoord,\ycoord)};


\addplot [color=mycolor3, line width=\lineWidth, forget plot]
table [x index=0, y index=1, col sep=comma]{img/DataCSV/Fig1b_bis_tol1.000000e-06.csv};

\addplot [color=mycolor3, densely dotted, line width=\lineWidth]
table [row sep=crcr]{
1 1e-6 \\
1944 1e-6 \\
};


\pgfplotstableread[col sep=comma]{img/DataCSV/Fig1b_bis_QP.csv}\datatable

\pgfplotstablegetelem{2}{x1}\of{\datatable}
\let\xcoord\pgfplotsretval

\pgfplotstablegetelem{2}{y1}\of{\datatable}
\let\ycoord\pgfplotsretval

\addplot [color=mycolor3, line width=\lineWidth, only marks, mark=o, mark options={solid, mycolor3}, forget plot]
  coordinates {(\xcoord,\ycoord)};


\addplot [color=mycolor4, line width=\lineWidth, forget plot]
table [x index=0, y index=1, col sep=comma]{img/DataCSV/Fig1b_bis_tol1.000000e-04.csv};

\addplot [color=mycolor4, loosely dashdotted, line width=\lineWidth]
table [row sep=crcr]{
1 1e-4 \\
1944 1e-4 \\
};


\pgfplotstableread[col sep=comma]{img/DataCSV/Fig1b_bis_QP.csv}\datatable

\pgfplotstablegetelem{1}{x1}\of{\datatable}
\let\xcoord\pgfplotsretval

\pgfplotstablegetelem{1}{y1}\of{\datatable}
\let\ycoord\pgfplotsretval

\addplot [color=mycolor4, line width=\lineWidth, only marks, mark=o, mark options={solid, mycolor4}, forget plot]
  coordinates {(\xcoord,\ycoord)};


\addplot [color=mycolor2, line width=\lineWidth, forget plot]
table [x index=0, y index=1, col sep=comma]{img/DataCSV/Fig1b_bis_tol1.000000e-02.csv};

\addplot [color=mycolor2, densely dashdotted, line width=\lineWidth]
table [row sep=crcr]{
1 1e-2 \\
1944 1e-2 \\
};


\pgfplotstableread[col sep=comma]{img/DataCSV/Fig1b_bis_QP.csv}\datatable

\pgfplotstablegetelem{0}{x1}\of{\datatable}
\let\xcoord\pgfplotsretval

\pgfplotstablegetelem{0}{y1}\of{\datatable}
\let\ycoord\pgfplotsretval

\addplot [color=mycolor2, line width=\lineWidth, only marks, mark=o, mark options={solid, mycolor2}, forget plot]
  coordinates {(\xcoord,\ycoord)};

\end{axis}
\end{tikzpicture}%
\label{subfig:1bis:b}}
    \subfigure[Integration profile for $T=100$, $\inp(t)=\sin(t)$ and $\inp(t)=\sin(10\,t)$]{
%
\tikzexternaldisable
\begin{tikzpicture}

\begin{axis}[%
width=0.64*\imageWidth,
	height=\imageHeight,
	scale only axis,
	scaled ticks=false,
	grid=both,
	grid style={line width=.1pt, draw=gray!10},
	major grid style={line width=.2pt,draw=gray!50},
	axis lines*=left,
	axis line style={line width=\lineWidth},
xmin=-0.2,
xmax=0.15,
xlabel style={font=\color{white!15!black}},
xlabel={$\Re(\lapVar)$},
ymin=-20,
ymax=27,
ylabel style={font=\color{white!15!black}},
ylabel={$ \imagunit\Im(\lapVar)$},
	axis background/.style={fill=white},
	legend style={%
		legend cell align=left, 
		align=left, 
		font=\tiny,
		draw=white!15!black,
		at={(1.00,0.98)},
		anchor=north east,},
]
\addplot [color=mycolor1, dashed, line width=\lineWidth, mark=none, mark options={solid, mycolor1}]
table [x index=0, y index=1, col sep=comma]{img/DataCSV/Fig1c_bis_1.csv};

\addlegendentry{$\Gamma$ for $\inp(t)=\sin(t)$}

\addplot [color=mycolor2, line width=\lineWidth, mark options={solid, mycolor2}]
table [x index=0, y index=1, col sep=comma]{img/DataCSV/Fig1c_bis_10.csv};

\addlegendentry{$\Gamma$ for $\inp(t)=\sin(10\,t)$}

\end{axis}
\end{tikzpicture}%
\label{subfig:1bis:c}}
 }
	\caption{Constrained mass-spring-damper system. Decay of the quadrature error for different values of the target accuracy $\tol$ (left and center). Truncated integration profile (right).}
	\label{fig2:MSD}%
\end{figure}%
The final numerical test for this example examines the performance of the method when applied to oscillatory and periodic input functions. For this experiment, we set $T = 100$ and $g = 500$, and consider two test cases:
\[
    \inp(t) = \sin(t) \quad \text{and} \quad \inp(t) = \sin(10\,t).
\]
\noindent Recall that the Laplace transform of $\sin(\alpha \,t)$, with $\alpha \in \mathbb{R}$, possesses poles at $\pm i\alpha$. Consequently, the integration contour $\Gamma$ must be constructed so that these poles lie to its right; see \Cref{subfig:1bis:c}. As $\alpha$ increases, the integration profile must enclose a larger region of the complex plane, and a direct consequence of this is that a larger number of quadrature points is required to achieve a prescribed tolerance~$\tol$. This behavior is confirmed by the results in \Cref{subfig:1bis:a,subfig:1bis:b}. Also, compared to the previously considered input functions, which lacked poles on the imaginary axis, we observe that a substantially larger number of quadrature points is needed. Nevertheless, we emphasize two aspects:
\begin{enumerate}
    \item the evaluation of the associated linear systems can be performed in parallel, mitigating the computational cost;
    
    \item for problems involving oscillatory input functions, more sophisticated quadrature rules could be employed to reduce the number of linear system evaluations. Also, since such oscillatory functions are analytic, their regularity can be exploited to accelerate the decay of the quadrature error with respect to the number of quadrature points, following the idea exploited in~\cite{HorG24}.
\end{enumerate}

\subsubsection{A \CIM for the Stokes problem} \label{subsec:num:Stokes}
    The instationary Stokes equations describe the flow of fluids at very low velocities without convection and coincide with the linearization of the Navier-Stokes equations around the zero-state. After a semi-discretization in space (see \cite{morwiki_stokes} which is based on \cite{Sch07}), we obtain the differential-algebraic system
\begin{align}\label{eq42}
	\left\{\quad
	\begin{aligned}
		\dot{\fv}(t) &= \fA_{11}\fv(t) + \fA_{12}\frho(t) + \fB_1\inp(t),\\
		\zeroVec &= \fA^\T_{12}\fv(t) + \fB_2\inp(t)
	\end{aligned}\right.
\end{align}
where $\fv(t)\in\R^{\stateDim_{\fv}}$ and $\frho(t)\in\R^{\stateDim_{\frho}}$ are the semidiscretized vectors of velocities and pressures, respectively. The \DAE~\eqref{eq42} has index two and the dimension $\stateDim = \stateDim_{\fv} + \stateDim_{\frho}$ of the system \eqref{eq42} depends on the fineness of the discretization and is usually large. The representation of system \eqref{eq42} in the form \eqref{eqn:DAE} reads as
\begin{align}
	\label{eqn:ref:mat:ST}
	\fE &= \begin{bmatrix}
		\fI & \zeroMat \\
		\zeroMat & \zeroMat
	\end{bmatrix}, &
	\fA &= \begin{bmatrix}
		\fA_{11} & \fA_{12} \\
		\fA_{12}^\T & \zeroMat
	\end{bmatrix}, &
	\source(t) &= \begin{bmatrix}
		\fB_{1} \\
		\fB_{2}
	\end{bmatrix}\inp(t).
\end{align}
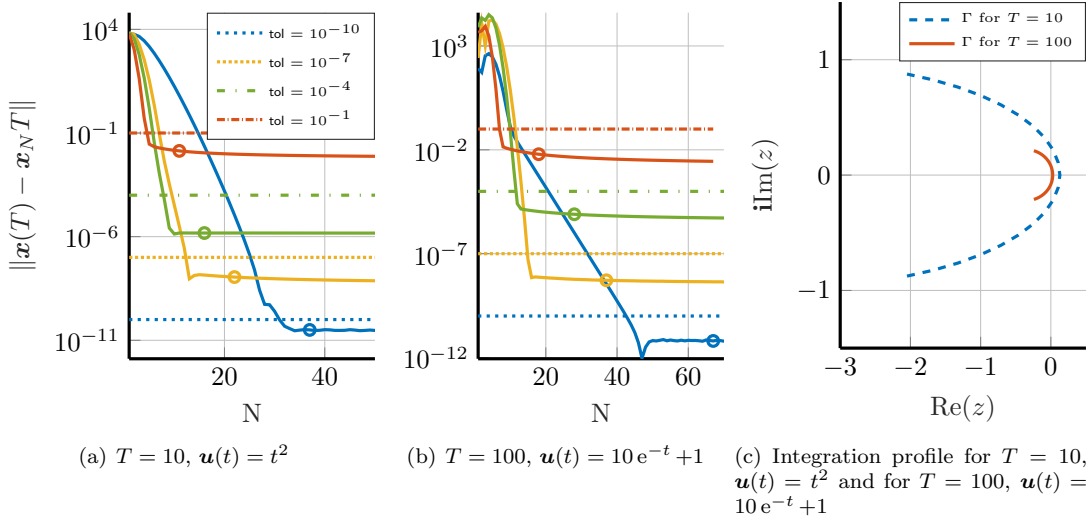
\begin{figure}[t]
	\centering{
	\subfigure[$T=10$, $\inp(t)=t^2$]{
%
\tikzexternaldisable
\begin{tikzpicture}

\begin{axis}[%
	width=0.64*\imageWidth,
	height=\imageHeight,
	scale only axis,
	scaled ticks=false,
	grid=both,
	grid style={line width=.1pt, draw=gray!10},
	major grid style={line width=.2pt,draw=gray!50},
	axis lines*=left,
	axis line style={line width=\lineWidth},
xmin=1,
xmax=50,
xlabel style={font=\color{white!15!black}},
xlabel={N},
ymode=log,
ymin=1.28412403997057e-12,
ymax=63860.480072029,
yminorticks=true,
ylabel style={font=\color{white!15!black}},
ylabel={$\|\stx(T)-\stx_N{T}\|$},
	axis background/.style={fill=white},
	legend style={%
		legend cell align=left, 
		align=left, 
		font=\tiny,
		draw=white!15!black,
		at={(1.00,1.00)},
		anchor=north east,},
]


\addplot [color=mycolor1, line width=\lineWidth, forget plot]
table [x index=0, y index=1, col sep=comma]{img/DataCSV/Fig2a_tol1.000000e-10.csv};

\addplot [color=mycolor1, dotted, line width=\lineWidth]
table [x index=0, y index=2, col sep=comma]{img/DataCSV/Fig2a_tol1.000000e-10.csv};

\addlegendentry{$\tol=10^{-10}$}

\pgfplotstableread[col sep=comma]{img/DataCSV/Fig2a_QP.csv}\datatable

\pgfplotstablegetelem{3}{x1}\of{\datatable}
\let\xcoord\pgfplotsretval

\pgfplotstablegetelem{3}{y1}\of{\datatable}
\let\ycoord\pgfplotsretval

\addplot [color=mycolor1, line width=\lineWidth, only marks, mark=o, mark options={solid, mycolor1}, forget plot]
 coordinates {(\xcoord,\ycoord)};


\addplot [color=mycolor3, line width=\lineWidth, forget plot]
table [x index=0, y index=1, col sep=comma]{img/DataCSV/Fig2a_tol1.000000e-07.csv};

\addplot [color=mycolor3, densely dotted, line width=\lineWidth]
table [x index=0, y index=2, col sep=comma]{img/DataCSV/Fig2a_tol1.000000e-07.csv};

\addlegendentry{$\tol=10^{-7}$}

\pgfplotstableread[col sep=comma]{img/DataCSV/Fig2a_QP.csv}\datatable

\pgfplotstablegetelem{2}{x1}\of{\datatable}
\let\xcoord\pgfplotsretval

\pgfplotstablegetelem{2}{y1}\of{\datatable}
\let\ycoord\pgfplotsretval

\addplot [color=mycolor3, line width=\lineWidth, only marks, mark=o, mark options={solid, mycolor3}, forget plot]
  coordinates {(\xcoord,\ycoord)};


\addplot [color=mycolor4, line width=\lineWidth, forget plot]
table [x index=0, y index=1, col sep=comma]{img/DataCSV/Fig2a_tol1.000000e-04.csv};

\addplot [color=mycolor4, loosely dashdotted, line width=\lineWidth]
table [x index=0, y index=2, col sep=comma]{img/DataCSV/Fig2a_tol1.000000e-04.csv};

\addlegendentry{$\tol=10^{-4}$}

\pgfplotstableread[col sep=comma]{img/DataCSV/Fig2a_QP.csv}\datatable

\pgfplotstablegetelem{1}{x1}\of{\datatable}
\let\xcoord\pgfplotsretval

\pgfplotstablegetelem{1}{y1}\of{\datatable}
\let\ycoord\pgfplotsretval

\addplot [color=mycolor4, line width=\lineWidth, only marks, mark=o, mark options={solid, mycolor4}, forget plot]
  coordinates {(\xcoord,\ycoord)};


\addplot [color=mycolor2, line width=\lineWidth, forget plot]
table [x index=0, y index=1, col sep=comma]{img/DataCSV/Fig2a_tol1.000000e-01.csv};

\addplot [color=mycolor2, densely dashdotted, line width=\lineWidth]
table [x index=0, y index=2, col sep=comma]{img/DataCSV/Fig2a_tol1.000000e-01.csv};

\addlegendentry{$\tol=10^{-1}$}

\pgfplotstableread[col sep=comma]{img/DataCSV/Fig2a_QP.csv}\datatable

\pgfplotstablegetelem{0}{x1}\of{\datatable}
\let\xcoord\pgfplotsretval

\pgfplotstablegetelem{0}{y1}\of{\datatable}
\let\ycoord\pgfplotsretval

\addplot [color=mycolor2, line width=\lineWidth, only marks, mark=o, mark options={solid, mycolor2}, forget plot]
  coordinates {(\xcoord,\ycoord)};

\end{axis}
\end{tikzpicture}%
\label{subfig:2a}}
	\subfigure[$T=100$, $\inp(t)=10\ee^{-t}+1$]{
%
\tikzexternaldisable
\begin{tikzpicture}

\begin{axis}[%
width=0.64*\imageWidth,
	height=\imageHeight,
	scale only axis,
	scaled ticks=false,
	grid=both,
	grid style={line width=.1pt, draw=gray!10},
	major grid style={line width=.2pt,draw=gray!50},
	axis lines*=left,
	axis line style={line width=\lineWidth},
xmin=1,
xmax=70,
xlabel style={font=\color{white!15!black}},
xlabel={N},
ymode=log,
ymin=8.75422047243395e-13,
ymax=4e4,
yminorticks=true,
ylabel style={font=\color{white!15!black}},
	axis background/.style={fill=white},
	legend style={%
		legend cell align=left, 
		align=left, 
		font=\tiny,
		draw=white!15!black,
		at={(1.00,1.00)},
		anchor=north east,},
]


\addplot [color=mycolor1, line width=\lineWidth, forget plot]
table [x index=0, y index=1, col sep=comma]{img/DataCSV/Fig2b_tol1.000000e-10.csv};

\addplot [color=mycolor1, dotted, line width=\lineWidth]
table [x index=0, y index=2, col sep=comma]{img/DataCSV/Fig2b_tol1.000000e-10.csv};


\pgfplotstableread[col sep=comma]{img/DataCSV/Fig2b_QP.csv}\datatable

\pgfplotstablegetelem{3}{x1}\of{\datatable}
\let\xcoord\pgfplotsretval

\pgfplotstablegetelem{3}{y1}\of{\datatable}
\let\ycoord\pgfplotsretval

\addplot [color=mycolor1, line width=\lineWidth, only marks, mark=o, mark options={solid, mycolor1}, forget plot]
 coordinates {(\xcoord,\ycoord)};


\addplot [color=mycolor3, line width=\lineWidth, forget plot]
table [x index=0, y index=1, col sep=comma]{img/DataCSV/Fig2b_tol1.000000e-07.csv};

\addplot [color=mycolor3, densely dotted, line width=\lineWidth]
table [x index=0, y index=2, col sep=comma]{img/DataCSV/Fig2b_tol1.000000e-07.csv};


\pgfplotstableread[col sep=comma]{img/DataCSV/Fig2b_QP.csv}\datatable

\pgfplotstablegetelem{2}{x1}\of{\datatable}
\let\xcoord\pgfplotsretval

\pgfplotstablegetelem{2}{y1}\of{\datatable}
\let\ycoord\pgfplotsretval

\addplot [color=mycolor3, line width=\lineWidth, only marks, mark=o, mark options={solid, mycolor3}, forget plot]
  coordinates {(\xcoord,\ycoord)};


\addplot [color=mycolor4, line width=\lineWidth, forget plot]
table [x index=0, y index=1, col sep=comma]{img/DataCSV/Fig2b_tol1.000000e-04.csv};

\addplot [color=mycolor4, loosely dashdotted, line width=\lineWidth]
table [x index=0, y index=2, col sep=comma]{img/DataCSV/Fig2b_tol1.000000e-04.csv};


\pgfplotstableread[col sep=comma]{img/DataCSV/Fig2b_QP.csv}\datatable

\pgfplotstablegetelem{1}{x1}\of{\datatable}
\let\xcoord\pgfplotsretval

\pgfplotstablegetelem{1}{y1}\of{\datatable}
\let\ycoord\pgfplotsretval

\addplot [color=mycolor4, line width=\lineWidth, only marks, mark=o, mark options={solid, mycolor4}, forget plot]
  coordinates {(\xcoord,\ycoord)};


\addplot [color=mycolor2, line width=\lineWidth, forget plot]
table [x index=0, y index=1, col sep=comma]{img/DataCSV/Fig2b_tol1.000000e-01.csv};

\addplot [color=mycolor2, densely dashdotted, line width=\lineWidth]
table [x index=0, y index=2, col sep=comma]{img/DataCSV/Fig2b_tol1.000000e-01.csv};


\pgfplotstableread[col sep=comma]{img/DataCSV/Fig2b_QP.csv}\datatable

\pgfplotstablegetelem{0}{x1}\of{\datatable}
\let\xcoord\pgfplotsretval

\pgfplotstablegetelem{0}{y1}\of{\datatable}
\let\ycoord\pgfplotsretval

\addplot [color=mycolor2, line width=\lineWidth, only marks, mark=o, mark options={solid, mycolor2}, forget plot]
  coordinates {(\xcoord,\ycoord)};

\end{axis}
\end{tikzpicture}%
\label{subfig:2b}}
    \subfigure[Integration profile for $T=10$, $\inp(t)=t^2$ and for $T=100$, $\inp(t)=10\ee^{-t}+1$]{
%
\tikzexternaldisable
\begin{tikzpicture}

\begin{axis}[%
width=0.64*\imageWidth,
	height=\imageHeight,
	scale only axis,
	scaled ticks=false,
	grid=both,
	grid style={line width=.1pt, draw=gray!10},
	major grid style={line width=.2pt,draw=gray!50},
	axis lines*=left,
	axis line style={line width=\lineWidth},
xmin=-3,
xmax=0.5,
xlabel style={font=\color{white!15!black}},
xlabel={$\Re(\lapVar)$},
ymin=-1.5,
ymax=1.5,
ylabel style={font=\color{white!15!black}},
ylabel={$ \imagunit\Im(\lapVar)$},
	axis background/.style={fill=white},
	legend style={%
		legend cell align=left, 
		align=left, 
		font=\tiny,
		draw=white!15!black,
		at={(1.00,1.00)},
		anchor=north east,},
]
\addplot [color=mycolor1, dashed, line width=\lineWidth, mark=none, mark options={solid, mycolor1}]
table [x index=0, y index=1, col sep=comma]{img/DataCSV/Fig2c_10.csv};

\addlegendentry{$\Gamma$ for $T=10$}

\addplot [color=mycolor2, line width=\lineWidth, mark options={solid, mycolor2}]
table [x index=0, y index=1, col sep=comma]{img/DataCSV/Fig2c_100.csv};

\addlegendentry{$\Gamma$ for $T=100$}

\end{axis}
\end{tikzpicture}%
\label{subfig:2c}}
 }
	\caption{Instationary Stokes problem. Decay of the quadrature error for different values of the target accuracy $\tol$ (left and center). Truncated integration profile (right).}
	\label{fig1:Stokes}%
\end{figure}%
and we set $\stateDim=1159$. \Cref{subfig:2a} and \Cref{subfig:2b} report the decay of the quadrature error for two different final times $T$ and input functions $\inp(t)$. The plots show that an approximation of the solution is always achieved with the prescribed precision $\tol$ and that the estimated number of quadrature points required, highlighted by the circles in the plots, is reliable and not too distant from the optimal one. In \Cref{subfig:2c}, we show two of the computed integration profiles related to \Cref{subfig:2a} and \Cref{subfig:2b}, after truncation. It is evident how the final time $T$ affects the shape of the profile.

\subsection{Numerics for the structured \texorpdfstring{$\eps$}{TEXT}-distance from singularity for parametric matrices} \label{subsec:num:ex:stab:rad}
We illustrate the performance of \Cref{alg:strcatured:eps:distance:singularity} on three benchmark problems: a semidiscrete Stokes problem depending on a single parameter, a two-parameter system derived from the Black–Scholes model, and a finite-difference discretization of a two-parameter convection–diffusion \PDE.

\subsubsection{The parametric Stokes problem} \label{sec:Stokes:example}

Consider the semi-discretized Stokes equations of \Cref{subsec:num:Stokes}, i.e.,
\begin{align}\label{eq42:bis}
	\left\{\quad
	\begin{aligned}
		\dot{\fv}(t) &= \mu\fA_{11}\fv(t) + \fA_{12}\frho(t) + \fB_1\inp(t),\\
		\zeroVec &= \fA^\T_{12}\fv(t) + \fB_2\inp(t)
	\end{aligned}\right.
\end{align}
where $\mu\in\R$ is the parameter that regulates the diffusion strength. With  $\mu=1$ we recover the test problem \eqref{eq42}. For this case, the matrix $\fA$ in \eqref{eqn:ref:mat:ST} depends on the
parameter $\mu$ and has the following structure
\begin{equation*}
\fA(\mu) 
\;=\; \mu_1 \hat{\fA}_1 + \hat{\fA}_2+\hat{\fA}_2^\T,
\end{equation*} 
where
\begin{equation}\label{eq:A1A2}
\hat{\fA}_1 \;\vcentcolon=\;\begin{bmatrix}
		\fA_{11} & \zeroMat \\
		\zeroMat & \zeroMat
	\end{bmatrix},\quad\hat{\fA}_2 \;\vcentcolon=\;\begin{bmatrix}
		\zeroMat& \fA_{12}  \\
		\zeroMat & \zeroMat
	\end{bmatrix},
\end{equation} 
while $\fE$ is as in \eqref{eqn:ref:mat:ST}. Thus, the assumption \eqref{eqn:CS:par:mat} is satisfied by the problem. Motivated by \Cref{sec:str-uns}, for the parameter value $\mu_0=1$, we consider the integration profile $\Gamma_{\mu_0}$ used in \Cref{subsec:num:Stokes} for $\tol=10^{-7}$, $T=10$, and $\stateDim=1159$. Our objective is to determine the parametric set $\hat{\prmtrSet}$ in which, for $\varepsilon=10^{-3}$, we have

\[
\|(\lapVar_j\fE-\fA(\mu))^{-1}\|\le\frac{1}{\varepsilon},\quad \text{for all } \mu\in\hat{\prmtrSet},\quad \text{and } j=1,\ldots,N-1;
\] 
being $\lapVar_j$ a quadrature point on $\Gamma_{\mu_0}$. We run \Cref{alg1} for each quadrature point $z_j$, and thus get
\begin{equation*}
\fDelta_j = \delta_{\varepsilon} \fL^{\cS_{\fA_1}}, \qquad  \fL^{\cS_{\fA_1}} = c_1\hat{\fA}_1,\qquad c_1\vcentcolon=\|\hat{\fA}_1\|_{\Frob}^{-1};
\end{equation*}
for which we expect 
\begin{equation}\label{eqn:cond:set}
    \|(z\fE-\fA(\mu))^{-1}\|\;\le\frac{1}{\varepsilon},\quad\text{for all}\quad \mu\in\hat \prmtrSet\;\vcentcolon=\;\{\mu\;|\;\mu\le \mu_0+c_1\delta\;=\vcentcolon\;\hat \mu\}.
\end{equation}

\begin{figure}[t]
	\centering{
	\subfigure[$\lapVar=0.42+0.05\imagunit$ and $\hat{\mu}=4.55$]{
%
\tikzexternaldisable
\begin{tikzpicture}

\begin{axis}[%
width=0.64*\imageWidth,
	height=\imageHeight,
	scale only axis,
	scaled ticks=false,
	grid=both,
	grid style={line width=.1pt, draw=gray!10},
	major grid style={line width=.2pt,draw=gray!50},
	axis lines*=left,
	axis line style={line width=\lineWidth},
xmode=log,
xmin=0.0001,
xmax=100,
xlabel style={font=\color{white!15!black}},
xlabel={$\mu$},
ymode=log,
ymin=1,
ymax=1e5,
yminorticks=true,
ylabel style={font=\color{white!15!black}},
axis background/.style={fill=white},
	legend style={%
		legend cell align=left, 
		align=left, 
		font=\tiny,
		draw=white!15!black,
		at={(0.95,1.00)},
		anchor=north east,},
]
\addplot [color=mycolor1, line width=\lineWidth]
  table [x index=0, y index=1, col sep=comma]{img/DataCSV/Fig3.csv};
  
\addlegendentry{$\|(\lapVar\mathbf{E}-\mathbf{A}(\mu))^{-1}\|$}

\addplot [color=mycolor2, dashed, line width=\lineWidth]
table [x index=0, y index=4, col sep=comma]{img/DataCSV/Fig3.csv};
\addlegendentry{$\varepsilon^{-1}$}

\addplot [color=black, dashdotted, line width=\lineWidth]
  table[row sep=crcr]{%
4.5498	1\\
4.5498	1e5\\
};
\addlegendentry{$\hat{\mu}$ from \Cref{alg1}}

\end{axis}
\end{tikzpicture}%
\label{subfig:3a}} 
        \hfil
	\subfigure[$\lapVar=-0.13 + 0.94\imagunit$ and $\hat{\mu}=4.56$]{
%
\tikzexternaldisable
\begin{tikzpicture}

\begin{axis}[%
width=0.64*\imageWidth,
	height=\imageHeight,
	scale only axis,
	scaled ticks=false,
	grid=both,
	grid style={line width=.1pt, draw=gray!10},
	major grid style={line width=.2pt,draw=gray!50},
	axis lines*=left,
	axis line style={line width=\lineWidth},
xmode=log,
xmin=0.0001,
xmax=100,
xlabel style={font=\color{white!15!black}},
xlabel={$\mu$},
ymode=log,
ymin=1,
ymax=1e5,
yminorticks=true,
ylabel style={font=\color{white!15!black}},
axis background/.style={fill=white},
	legend style={%
		legend cell align=left, 
		align=left, 
		font=\tiny,
		draw=white!15!black,
		at={(1.00,1.00)},
		anchor=north east,},
]
\addplot [color=mycolor1, line width=\lineWidth]
  table [x index=0, y index=1, col sep=comma]{img/DataCSV/Fig3.csv};

\addplot [color=mycolor2, dashed, line width=\lineWidth]
table [x index=0, y index=4, col sep=comma]{img/DataCSV/Fig3.csv};

\addplot [color=black, dashdotted, line width=\lineWidth]
  table[row sep=crcr]{%
4.5498	1\\
4.5498	1e5\\
};

\end{axis}
\end{tikzpicture}%
\label{subfig:3b}}
        \hfil
    \subfigure[$\lapVar=-1.64 + 1.71\imagunit$ and $\hat{\mu}=4.58$]{
%
\tikzexternaldisable
\begin{tikzpicture}

\begin{axis}[%
width=0.64*\imageWidth,
	height=\imageHeight,
	scale only axis,
	scaled ticks=false,
	grid=both,
	grid style={line width=.1pt, draw=gray!10},
	major grid style={line width=.2pt,draw=gray!50},
	axis lines*=left,
	axis line style={line width=\lineWidth},
xmode=log,
xmin=0.0001,
xmax=100,
xlabel style={font=\color{white!15!black}},
xlabel={$\mu$},
ymode=log,
ymin=1,
ymax=1e5,
yminorticks=true,
ylabel style={font=\color{white!15!black}},
axis background/.style={fill=white},
	legend style={%
		legend cell align=left, 
		align=left, 
		font=\tiny,
		draw=white!15!black,
		at={(1.00,1.00)},
		anchor=north east,},
]
\addplot [color=mycolor1, line width=\lineWidth]
  table [x index=0, y index=1, col sep=comma]{img/DataCSV/Fig3.csv};

\addplot [color=mycolor2, dashed, line width=\lineWidth]
table [x index=0, y index=4, col sep=comma]{img/DataCSV/Fig3.csv};

\addplot [color=black, dashdotted, line width=\lineWidth]
  table[row sep=crcr]{%
4.5498	1\\
4.5498	1e5\\
};

\end{axis}
\end{tikzpicture}%
\label{subfig:3c}}
 }
	\caption{Instationary Stokes problem: structured stability radius $\delta=c_1^{-1}(\hat{\mu}-\mu_0)$ computed towards \Cref{alg1}, for $\mu_0=1$ and $c_1=\|\hat{\fA}_1\|_{\Frob}^{-1}$, and fixed unstructured stability radius $\varepsilon=10^{-3}$.}
	\label{fig3:Stokes}%
\end{figure}
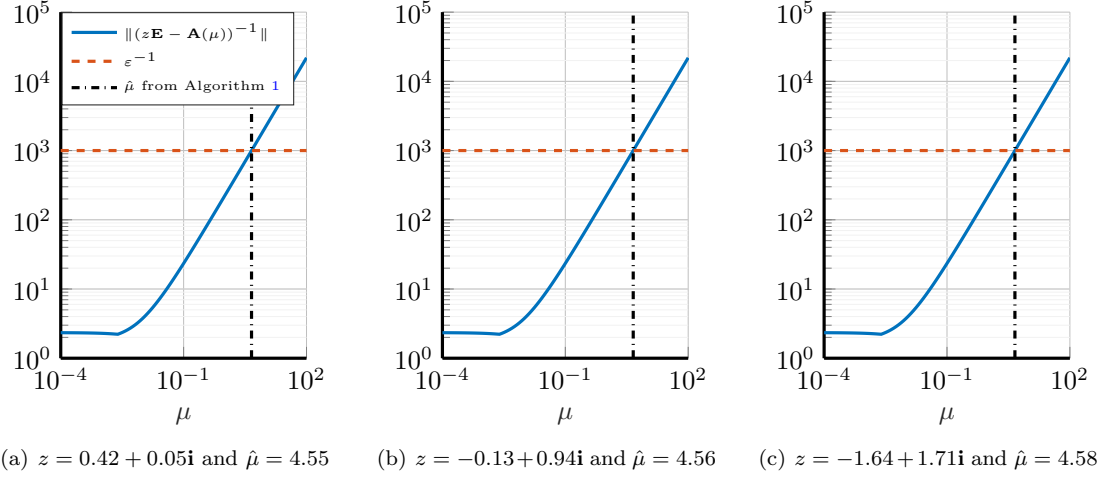%

To illustrate the corresponding numerical results, we select the extreme quadrature points along with one central point, and for each of these we show in \Cref{fig3:Stokes} the value of $\hat \mu$ computed by \Cref{alg1}, together with the generalized resolvent norm over a broad range of parameters. The fact that, in the plots, the values of $\hat \mu$ intersect the level $\varepsilon$ precisely at points lying on the $\|(z\fE-\fA(\mu))^{-1}\|$ curve demonstrates that \Cref{alg1} is capable of determining the exact global minima for this problem.

\begin{remark}
We observe that the value $\hat\mu$ does not change significantly with respect to the chosen $z$, indicating a certain robustness of the resolvent norm for this problem at the quadrature points. This could be exploited for computations in the following way: when running \Cref{alg1} for $z_j$, one could initiate the structured and unstructured perturbation $\fL^{\cS}$ and $\fL$ with the one given by \Cref{alg1} for $z_{j-1}$. 
\end{remark}

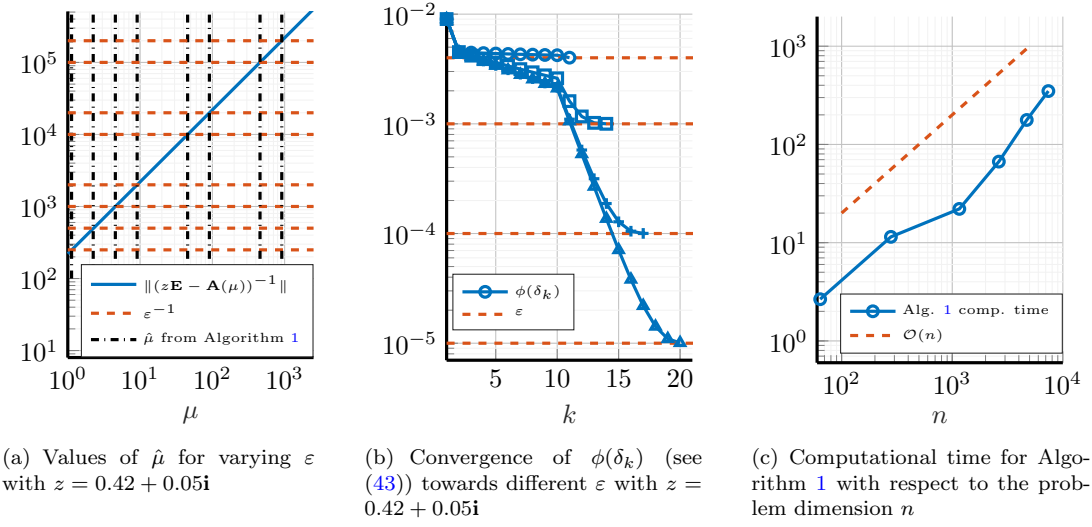
\begin{figure}[t]
	\centering{
	\subfigure[Values of $\hat \mu$ for varying $\varepsilon$ with $\lapVar=0.42+0.05\imagunit$]{
		\tikzexternaldisable
\begin{tikzpicture}

\begin{axis}[%
width=0.64*\imageWidth,
	height=\imageHeight,
	scale only axis,
	scaled ticks=false,
	grid=both,
	grid style={line width=.1pt, draw=gray!10},
	major grid style={line width=.2pt,draw=gray!50},
	axis lines*=left,
	axis line style={line width=\lineWidth},
xmode=log,
xmin=1,
xmax=2511,
xminorticks=true,
xlabel style={font=\color{white!15!black}},
xlabel={$\mu$},
ymode=log,
ymin=8,
ymax=5.1e5,
yminorticks=true,
axis background/.style={fill=white},
axis background/.style={fill=white},
	legend style={%
		legend cell align=left, 
		align=left, 
		font=\tiny,
		draw=white!15!black,
		at={(1.00,0.27)},
		anchor=north east,},
]
\addplot [color=mycolor1]
  [color=mycolor1, line width=\lineWidth]
  table [x index=0, y index=1, col sep=comma]{img/DataCSV/Fig4a.csv};
\addlegendentry{$\|(\lapVar\mathbf{E}-\mathbf{A}(\mu))^{-1}\|$}

\addplot [color=mycolor2, dashed, line width=\lineWidth]
  table[row sep=crcr]{%
1	250\\
2511.88643150958	250\\
};
\addlegendentry{$\varepsilon^{-1}$}

\addplot [color=black, dashdotted, line width=\lineWidth]
  table[row sep=crcr]{%
1.12789238115606	100\\
1.12789238115606	1000000\\
};
\addlegendentry{$\hat{\mu}$ from \Cref{alg1}}

\addplot [color=mycolor2, dashed, line width=\lineWidth, forget plot]
  table[row sep=crcr]{%
1	500\\
2511.88643150958	500\\
};
\addplot [color=black, dashdotted, line width=\lineWidth, forget plot]
  table[row sep=crcr]{%
2.24821532302093	100\\
2.24821532302093	1000000\\
};
\addplot [color=mycolor2, dashed, line width=\lineWidth, forget plot]
  table[row sep=crcr]{%
1	1000\\
2511.88643150958	1000\\
};
\addplot [color=black, dashdotted, line width=\lineWidth, forget plot]
  table[row sep=crcr]{%
4.54977569579335	100\\
4.54977569579335	1000000\\
};
\addplot [color=mycolor2, dashed, line width=\lineWidth, forget plot]
  table[row sep=crcr]{%
1	2000\\
2511.88643150958	2000\\
};
\addplot [color=black, dashdotted, line width=\lineWidth, forget plot]
  table[row sep=crcr]{%
9.11072546487305	100\\
9.11072546487305	1000000\\
};
\addplot [color=mycolor2, dashed, line width=\lineWidth, forget plot]
  table[row sep=crcr]{%
1	10000\\
2511.88643150958	10000\\
};
\addplot [color=black, dashdotted, line width=\lineWidth, forget plot]
  table[row sep=crcr]{%
45.4848150433512	100\\
45.4848150433512	1000000\\
};
\addplot [color=mycolor2, dashed, line width=\lineWidth, forget plot]
  table[row sep=crcr]{%
1	20000\\
2511.88643150958	20000\\
};
\addplot [color=black, dashdotted, line width=\lineWidth, forget plot]
  table[row sep=crcr]{%
90.9732304351844	100\\
90.9732304351844	1000000\\
};
\addplot [color=mycolor2, dashed, line width=\lineWidth, forget plot]
  table[row sep=crcr]{%
1	100000\\
2511.88643150958	100000\\
};
\addplot [color=black, dashdotted, line width=\lineWidth, forget plot]
  table[row sep=crcr]{%
456.375297776009	100\\
456.375297776009	1000000\\
};
\addplot [color=mycolor2, dashed, line width=\lineWidth, forget plot]
  table[row sep=crcr]{%
1	200000\\
2511.88643150958	200000\\
};
\addplot [color=black, dashdotted, line width=\lineWidth, forget plot]
  table[row sep=crcr]{%
912.763247586598	100\\
912.763247586598	1000000\\
};
\end{axis}
\end{tikzpicture}%
\label{subfig:4a}}
        \hfil
	\subfigure[Convergence of $\phi(\delta_k)$ (see \eqref{zero-delta}) towards different $\varepsilon$ with $\lapVar=0.42+0.05\imagunit$]{
		\tikzexternaldisable
\begin{tikzpicture}

\begin{axis}[%
width=0.64*\imageWidth,
	height=\imageHeight,
	scale only axis,
	scaled ticks=false,
	grid=both,
	grid style={line width=.1pt, draw=gray!10},
	major grid style={line width=.2pt,draw=gray!50},
	axis lines*=left,
	axis line style={line width=\lineWidth},
xmin=1,
xmax=21,
xminorticks=true,
xlabel style={font=\color{white!15!black}},
xlabel={$k$},
ymode=log,
ymin=7e-6,
ymax=1e-2,
yminorticks=true,
axis background/.style={fill=white},
axis background/.style={fill=white},
	legend style={%
		legend cell align=left, 
		align=left, 
		font=\tiny,
		draw=white!15!black,
		at={(0.52,0.25)},
		anchor=north east,},
]
\addplot [color=mycolor1,mark=o,mark options={solid, mycolor1}]
  [color=mycolor1, line width=\lineWidth]
  table [x index=0, y index=1, col sep=comma]{img/DataCSV/Fig4b/1.csv};
\addlegendentry{$\phi(\delta_k)$}

\addplot [color=mycolor2, dashed, line width=\lineWidth]
  table[row sep=crcr]{%
1	4e-3\\
40	4e-3\\
};
\addlegendentry{$\varepsilon$}

\addplot [color=mycolor1,mark=square,mark options={solid,mycolor1}]
  [color=mycolor1, line width=\lineWidth]
  table [x index=0, y index=1, col sep=comma]{img/DataCSV/Fig4b/3.csv};

\addplot [color=mycolor2, dashed, line width=\lineWidth]
  table[row sep=crcr]{%
1	1e-3\\
40	1e-3\\
};
\addplot [color=mycolor1,mark=+,mark options={solid, mycolor1}]
  [color=mycolor1, line width=\lineWidth]
  table [x index=0, y index=1, col sep=comma]{img/DataCSV/Fig4b/5.csv};

\addplot [color=mycolor2, dashed, line width=\lineWidth]
  table[row sep=crcr]{%
1	1e-4\\
40	1e-4\\
};
\addplot [color=mycolor1,mark=triangle,mark options={solid, mycolor1}]
  [color=mycolor1, line width=\lineWidth]
  table [x index=0, y index=1, col sep=comma]{img/DataCSV/Fig4b/7.csv};

\addplot [color=mycolor2, dashed, line width=\lineWidth]
  table[row sep=crcr]{%
1	1e-5\\
40	1e-5\\
};
\end{axis}
\end{tikzpicture}%
\label{subfig:4b}}
        \hfil
    \subfigure[Computational time for \Cref{alg1} with respect to the problem dimension $\stateDim$]{
		\tikzexternaldisable
\begin{tikzpicture}

\begin{axis}[%
width=0.64*\imageWidth,
	height=\imageHeight,
	scale only axis,
	scaled ticks=false,
	grid=both,
	grid style={line width=.1pt, draw=gray!10},
	major grid style={line width=.2pt,draw=gray!50},
	axis lines*=left,
	axis line style={line width=\lineWidth},
xmode=log,
xmin=6e1,
xmax=1e4,
xminorticks=true,
xlabel style={font=\color{white!15!black}},
xlabel={$\stateDim$},
ymode=log,
ymin=6e-1,
ymax=2e3,
yminorticks=true,
axis background/.style={fill=white},
axis background/.style={fill=white},
	legend style={%
		legend cell align=left, 
		align=left, 
		font=\tiny,
		draw=white!15!black,
		at={(1.0,0.20)},
		anchor=north east,},
]
\addplot [color=mycolor1,mark=o,mark options={solid, mycolor1}, line width=\lineWidth]
  table [x index=0, y index=1, col sep=comma]{img/DataCSV/Fig4c.csv};
\addlegendentry{Alg. \ref{alg1} comp. time}

\addplot [color=mycolor2, dashed, line width=\lineWidth]
  table [x index=2, y index=3, col sep=comma]{img/DataCSV/Fig4c.csv};
\addlegendentry{$\calO(\stateDim)$}

\end{axis}
\end{tikzpicture}%
\label{subfig:4c}}
 }
	\caption{Instationary Stokes problem: computed $\hat\mu$ via \Cref{alg1} for different $\varepsilon$ (left), convergence of $\phi(\delta_k)$ towards the outer iterations $k$ of \Cref{alg1} for different $\varepsilon$ (middle), and computational behavior with respect to $\stateDim$ of \Cref{alg1} (right).}
	\label{fig4:Stokes}%
\end{figure}%

Next, we fix $\lapVar=0.42+0.05\imagunit$ and show the computed $\hat \mu$ for different values of $\varepsilon$ in \Cref{subfig:4a}. \Cref{alg1} succeeds in determining the correct $\hat \mu$ under which \eqref{eqn:cond:set} is satisfied for different $\varepsilon$. The convergence behavior is shown in \Cref{subfig:4b}, where we show the decay of
\[
\phi(\delta_k)=\sigma_{\min}\bigl(\lapVar\fE-\fA(\mu_0)+ \delta_k \fL^\cS( \delta_k)\bigr)
\]
as a function of the outer iteration index $k$. We observe an exponential decrease until $\phi(\delta)$ reaches the target threshold $\varepsilon$. Finally, \Cref{subfig:4c} illustrates how the computational cost of \Cref{alg1} scales as a function of $\stateDim$. As discussed in \Cref{sec3:comp:aspects}, for sparse problems with sparsely structured perturbations, the anticipated linear dependence on $\stateDim$ is indeed confirmed.

\subsubsection{The parametric Black-Scholes problem} \label{sec:BS:prob}
 Next, we consider an \ODE system derived from the space discretization of the Black-Scholes operator \cite{BlaS73}, using the scheme proposed in \cite{HouW09}. The parameters $[\mu_1, \mu_2]=\vcentcolon\prmtr$ represent volatility and interest rate, respectively, and the matrices are
\begin{equation}\label{eqn:BS:aff}
	\fA(\prmtr) \;=\; \mu_1 \fA_1 + \mu_2 \fA_2, \qquad \fE\;=\;\fI_{\stateDim};
\end{equation}
where $\fA_1, \fA_2\in\R^{\stateDim\times\stateDim}$, $\stateDim = 5 \cdot 10^3$, are sparse. Following \Cref{sec:problem}, we work with the sets of structured matrices $\cS_{\fA_1}$ and $\cS_{\fA_2}$. Our goal is to determine a neighborhood of
a given $\prmtr_0$ such that, for a given $\varepsilon$, the norm $\|\lapVar\Id_{\stateDim}-\fA(\prmtr)\|$ is smaller than or equal to $\varepsilon^{-1}$ for all $\prmtr$ in this neighborhood when evaluated over a certain set of $\lapVar\in\C$, i.e., the quadrature points used for a \CIM.

Given $\varepsilon>0$, \Cref{alg1} returns $\delta_{\varepsilon}$, $\fL$, and $\tilde \fL_j\in\cS_{\fA_j}$ for $j=1,2$; from here, let us define 
\begin{equation}\label{eqn:local:set}
\hat \prmtrSet_j\;\vcentcolon=\;\left\{\mu_j\;|\;\mu_j\;\le\; \mu_{0,j}+\frac{\delta_{\varepsilon}}{\langle \tilde\fL_{j},\fA_j\rangle}\right\},
\end{equation}
then we want to verify that 
\begin{equation}\label{eqn:cond:set:2}
    \|(z\fE-\fA(\prmtr))^{-1}\|\;\le\frac{1}{\varepsilon},\quad\text{for all}\quad \prmtr\in\hat \prmtrSet\;\vcentcolon=\;\prod_{j=1}^{2}\hat \prmtrSet_j.
\end{equation}
Note that the set \eqref{eqn:cond:set:2} is unbounded, whereas \Cref{alg1} is designed for perturbation frameworks and, therefore, is based on local optimization. Consequently, in general, we should expect only \eqref{eqn:cond:set:2} to hold in a neighborhood of $\prmtr_0$, where part of the boundary is determined by the inequality in \eqref{eqn:local:set}, rather than over the entire unbounded domain $\hat \prmtrSet$. To determine the remaining parts of the boundaries, one must execute \Cref{alg1} for various choices of $\prmtr_0$ and then merge the resulting inequalities.

We set $\prmtr_0=(10^{-4},1)$ and construct the profile $\Gamma_{\prmtr_0}$ for $T=10$ and $\tol=10^{-6}$. \Cref{fig5} displays the results of applying \Cref{alg1} on two different values of $\varepsilon$.
\begin{figure}[t]
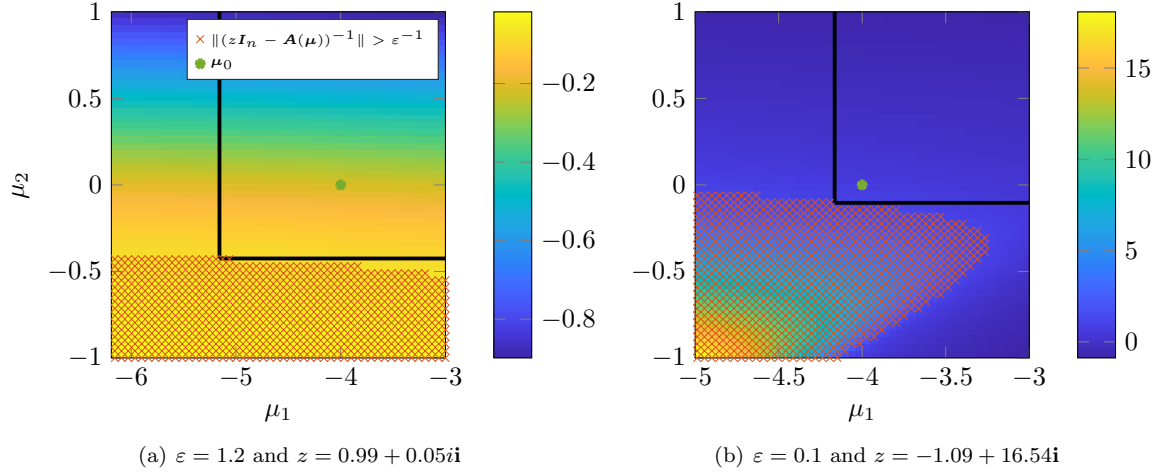

	\centering{
		\subfigure[$\varepsilon=1.2$ and $\lapVar=0.99+0.05i\imagunit$]{
			\input{img/Fig5a.tex}\label{subfig:5a}}
		\subfigure[$\varepsilon=0.1$ and $\lapVar=-1.09 +16.54\imagunit$]{
			\input{img/Fig5b.tex}\label{subfig:5b}}
	}
	\caption{Black-Scholes test problem: $\log_{10}\left(\|(\lapVar\fI_{\stateDim}-\fA(\prmtr))^{-1}\|\right)$ displayed over a $\log_{10}$ scale domain in both the $\mu_1$, $\mu_2$ axis. The black curves denote the boundaries determined by the inequalities in \eqref{eqn:local:set}, and the green asterisk marks the point $\prmtr_0=(10^{-4},10^0)$. Finally, red crosses are the parameter instances for which the desired resolvent bound is not satisfied.}\label{fig5}
\end{figure}%
In \Cref{subfig:5a}, we choose $\varepsilon=1.2$, and the boundary of the set $\hat \prmtrSet$, expressed by the black lines, correctly encloses the points that satisfy $\|(\lapVar\fI_{\stateDim}-\fA(\prmtr))^{-1}\|_2\le\varepsilon^{-1}$. Immediately beyond this boundary, we find the points for which $\|(\lapVar\fI_{\stateDim}-\fA(\prmtr))^{-1}\|_2>\varepsilon^{-1}$. In \Cref{subfig:5b}, we take $\varepsilon=10^{-1}$ and examine the last quadrature point on the contour $\Gamma_{\prmtr_0}$. Once more, we can observe that \Cref{alg1} accurately determines the boundary values at which the resolvent becomes excessively large.

We conclude with the observation that the parameter set satisfying $\|(\lapVar\fI_{\stateDim}-\fA(\prmtr))^{-1}\|_2\le10$ is larger than the set defined by $\|(\lapVar\fI_{\stateDim}-\fA(\prmtr))^{-1}\|_2\le1.2^{-1}$, however, from one quadrature point to the next, the bounds of $\hat \prmtr$ shrink markedly. In addition, the range of values assumed by the resolvent in \Cref{subfig:5b} is wider than in \Cref{subfig:5a}.This behavior can be attributed to the presence of one or more eigenvalues that are close to $\lapVar=-1.07 +16.45\imagunit$ for those $\prmtr$ that make $\|(\lapVar\fI_{\stateDim}-\fA(\prmtr))^{-1}\|_2$ large. 
\subsubsection{The parametric convection-diffusion problem} \label{sec:con:dif:prob}
Consider the linear convection-diffusion equation
\begin{equation}\label{CD}
	\frac{\partial u}{\partial t}\;=\; \mu_2 u_{hh}+ \mu_1 u_h,\qquad h \in [0,L],\;\;0<t\le T,
\end{equation}
with homogeneous Dirichlet boundary conditions.
The unknown function $u(h,t)$ depends on the diffusivity $\mu_2>0$ and the velocity $\mu_1>0$. 
Using a standard second-order finite-difference discretization for the diffusion term and a first-order upwind finite-difference discretization for the convection term, we construct the matrix $\fA(\prmtr)$, with $\prmtr\vcentcolon=[\mu_1,\mu_2]$. As a result, this leads to a linear system of \ODE with $\fE=\fI_{\stateDim}$. We denote by ${\rm trid}(a,b,c)$ a tridiagonal Toeplitz matrix whose lower diagonal entries are all equal to $a$, whose main diagonal entries are all equal to $b$, and whose upper diagonal entries are all equal to $c$. Next, we define the matrices that discretize
$u_{hh}$ and $u_{h}$ on a uniform grid with equally spaced grid points, each separated by a distance $\Delta h$, as
\[
\fA_1 \;=\; \frac{1}{\Delta h} {\rm trid} \left( -1, 1, 0 \right), \qquad \fA_2 \;=\; \frac{1}{\Delta h^2} {\rm trid} \left( 1, -2, 1 \right).
\]
The matrix $\fA(\prmtr)$ is then given by
\begin{equation*}
\fA(\prmtr) = \mu_1 \fA_1 + \mu_2 \fA_2,
\end{equation*} 
and we can consider the set of structured matrices $\cS_{\fA}$ for this problem as that given by the cartesian product of $\cS_{\fA_1}$ and $\cS_{\fA_2}$, following \Cref{sec:problem}. The same passages of \Cref{sec:BS:prob} can be repeated here to determine the set $\hat \prmtrSet$; see \eqref{eqn:local:set} and \eqref{eqn:cond:set:2}.

We set $\stateDim=10^4$, $T = 10$, and $\tol = 10^{-6}$ and construct the profile $\Gamma_{\prmtr_0}$ for $\prmtr_0 = (10^{-2}, 1)$. After executing \Cref{alg1} with $\varepsilon = 1$, we observe that the variation in the first allowable parameter is negligible, and therefore decided to display the results only for $\prmtr_2$. Thus, we determine the set $\hat \prmtrSet$ by varying $\mu_2$ while keeping $\mu_1$ fixed. \Cref{subfig:6a} and \ref{subfig:6b} display the results at two quadrature points, respectively. \Cref{alg1} successfully identifies the extreme value $\hat \mu_2$, and applying the inequality \eqref{eqn:local:set} for $j=2$, we can determine $\tilde \mu_2$ such that for any $\mu_2$ in the interval $[\hat \mu_2, \tilde \mu_2]$ (indicated by the two vertical black lines on the graphs), the inequality
\[
\|(\lapVar\fI_{\stateDim} - \fA(10^{-2}, \mu_2))^{-1}\|\; \le \;\frac{1}{\varepsilon},
\]
holds. 
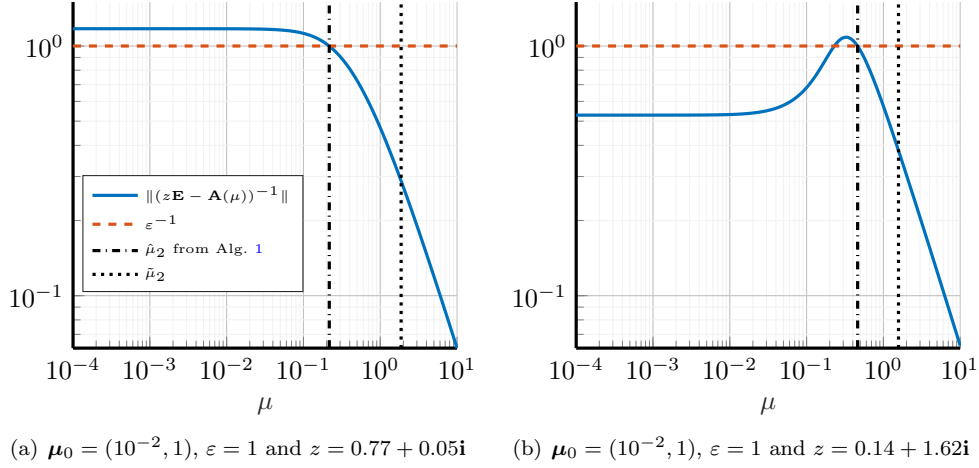
\begin{figure}[t]
	\centering{
		\subfigure[$\prmtr_0=(10^{-2},1)$, $\varepsilon=1$ and $\lapVar=0.77 + 0.05\imagunit$]{
%
\tikzexternaldisable
\begin{tikzpicture}

\begin{axis}[%
width=\imageWidth,
height=\imageHeight,
	scale only axis,
	scaled ticks=false,
	grid=both,
	grid style={line width=.1pt, draw=gray!10},
	major grid style={line width=.2pt,draw=gray!50},
	axis lines*=left,
	axis line style={line width=\lineWidth},
xmode=log,
xmin=0.0001,
xmax=10,
xlabel style={font=\color{white!15!black}},
xlabel={$\mu$},
ymode=log,
ymin=0.0616591491087206,
ymax=1.5,
yminorticks=true,
ylabel style={font=\color{white!15!black}},
axis background/.style={fill=white},
	legend style={%
		legend cell align=left, 
		align=left, 
		font=\tiny,
		draw=white!15!black,
		at={(0.60,0.50)},
		anchor=north east,},
]
\addplot [color=mycolor1, line width=\lineWidth]
  table [x index=0, y index=1, col sep=comma]{img/DataCSV/Fig6a.csv};
  
\addlegendentry{$\|(\lapVar\mathbf{E}-\mathbf{A}(\mu))^{-1}\|$}

\addplot [color=mycolor2, dashed, line width=\lineWidth]
table [x index=0, y index=3, col sep=comma]{img/DataCSV/Fig6a.csv};
\addlegendentry{$\varepsilon^{-1}$}

\addplot [color=black, dashdotted, line width=\lineWidth]
  table[row sep=crcr]{%
0.21698	0.0616591491087206\\
0.21698	1.5\\
};
\addlegendentry{$\hat{\mu}_2$ from Alg.~\ref{alg1}}

\addplot [color=black, dotted, line width=\lineWidth]
  table[row sep=crcr]{%
1.87	0.0616591491087206\\
1.87	1.5\\
};
\addlegendentry{$\tilde{\mu}_2$}

\end{axis}
\end{tikzpicture}%
\label{subfig:6a}}
		\subfigure[$\prmtr_0=(10^{-2},1)$, $\varepsilon=1$ and $\lapVar=0.14 + 1.62\imagunit$]{
%
\tikzexternaldisable
\begin{tikzpicture}

\begin{axis}[%
width=\imageWidth,
height=\imageHeight,
	scale only axis,
	scaled ticks=false,
	grid=both,
	grid style={line width=.1pt, draw=gray!10},
	major grid style={line width=.2pt,draw=gray!50},
	axis lines*=left,
	axis line style={line width=\lineWidth},
xmode=log,
xmin=0.0001,
xmax=10,
xlabel style={font=\color{white!15!black}},
xlabel={$\mu$},
ymode=log,
ymin=0.0616591491087206,
ymax=1.5,
yminorticks=true,
ylabel style={font=\color{white!15!black}},
axis background/.style={fill=white},
	legend style={%
		legend cell align=left, 
		align=left, 
		font=\tiny,
		draw=white!15!black,
		at={(0.62,0.40)},
		anchor=north east,},
]
\addplot [color=mycolor1, line width=\lineWidth]
  table [x index=0, y index=2, col sep=comma]{img/DataCSV/Fig6a.csv};
  

\addplot [color=mycolor2, dashed, line width=\lineWidth]
table [x index=0, y index=3, col sep=comma]{img/DataCSV/Fig6a.csv};

\addplot [color=black, dashdotted, line width=\lineWidth]
  table[row sep=crcr]{%
0.4622	0.0616591491087206\\
0.4622	1.5\\
};

\addplot [color=black, dotted, line width=\lineWidth]
  table[row sep=crcr]{%
1.573	0.0616591491087206\\
1.573	1.5\\
};

\end{axis}
\end{tikzpicture}%
\label{subfig:6b}}
	}
	\caption{Convection-diffusion test problem: the set $\hat{\prmtrSet}$ computed for $\varepsilon=1$, $\mu_1=10^{-2}$, and two different quadrature points.}\label{fig6}
\end{figure}%

\Cref{subfig:7a,subfig:7b} show the results obtained for the reference parameter $\prmtr_0=[10^{-4},1]$. In this case, we observed a non-negligible variation in $\prmtr_1$, and therefore we present the results in a two-dimensional plot. The results can be interpreted in a way analogous to those in \Cref{fig5}; once more, \Cref{alg1} successfully identifies the threshold value of the structured stability radius $\delta_{\varepsilon}$.

\begin{figure}[t]
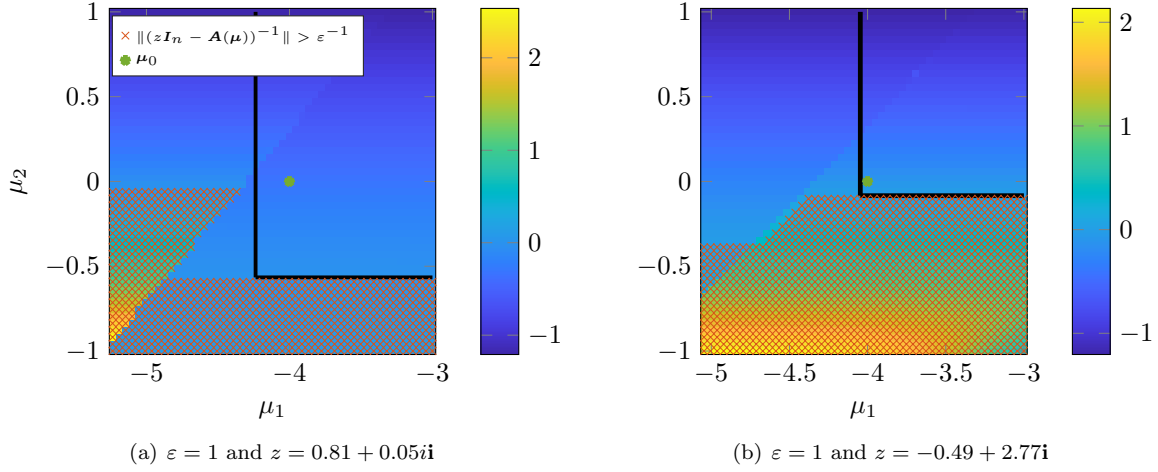

	\centering{
		\subfigure[$\varepsilon=1$ and $\lapVar=0.81+0.05i\imagunit$]{
			\input{img/Fig7a.tex}\label{subfig:7a}}
			\hfill
		\subfigure[$\varepsilon=1$ and $\lapVar=-0.49 +2.77\imagunit$]{
			\input{img/Fig7b.tex}\label{subfig:7b}}
	}
	\caption{Convection-diffusion test problem: $\log_{10}\left(\|(\lapVar\fI_{\stateDim}-\fA(\prmtr))^{-1}\|\right)$ displayed over a $\log_{10}$ scale domain in both the $\mu_1$, $\mu_2$ axis. The black curves denote the boundaries determined by the inequalities in \eqref{eqn:local:set}, and the green asterisk marks the point $\prmtr_0=(10^{-4},10^0)$. Finally, red crosses are the parameter instances for which the desired resolvent bound is not satisfied.}\label{fig7}
\end{figure}%

\section{Conclusion} \label{sec:conc}
The contribution of this work is twofold. First, we introduce a novel time-integration scheme for \DAE systems based on \CIM. A central step in evaluating the efficiency of the proposed method is the analysis of the generalized resolvent $\|(\lapVar\fE-\fA)^{-1}\|$ in the case where $\fE$ is singular. Although Proposition~\ref{lemma1} shows that this quantity behaves like the maximum between the resolvent norm of a general nonnormal matrix and a polynomial of degree $\indDAE-1$, where $\indDAE$ denotes the index of the \DAE, in \Cref{lemma2} we establish that the \CIM framework remains suitable for \DAEs provided that the right-hand side is sufficiently differentiable. These regularity requirements are analogous to the smoothness assumptions commonly imposed in the classical existence theory for \DAE solutions. The theoretical findings are subsequently corroborated by numerical experiments presented in \Cref{sec:num:CIM:DAE}. Second, we introduced a combined structured–unstructured eigenvalue perturbation framework to derive bounds on the distance to singularity of the generalized resolvent in the case where the matrix $\fA$ is given as an affine parametric function. An extension of this analysis to the matrix $\fE$ is also outlined in the remarks. By determining the largest admissible structured perturbation in the Frobenius norm, we can characterize a parametric region in which the generalized resolvent remains uniformly bounded by a prescribed value $\varepsilon^{-1}$. This property is crucial for the application of \CIM within the projection-based \MOR for parametric problems \cite{GugM23}, since these methods strongly rely on identifying a unique integration profile that is valid for an entire set of continuous parameters.

Several avenues for further investigation emerge as natural continuations of this work. To address \cref{issue1}, one may employ a combined structured–unstructured eigenvalue perturbation analysis to derive limits on the location of eigenvalues throughout the parametric domain. This development would constitute the final step toward a rigorous, globally valid assignment of a unique integration profile across a wide range of parameter values. In addition, it seems promising to generalize the \CIMs-\MOR framework proposed in \cite{GugM23} to the parametric setting \DAE. This extension, fully independent of time-stepping integrators, has the potential to facilitate the construction of projection spaces while bypassing the stabilization procedures or decoupling strategies commonly required in state-of-the-art parametric methods \DAE. Finally, in the contest of large scale problems, solving several eigenvalue problems as in \Cref{alg1} requires be prohibitive. To avoid this, it is possible to employ subspace projection methods, see, for instance, \cite{ManMG26}, to effectively reduce the size of the matrices involved in \Cref{alg1} and thus speed up the entire optimization.

\section*{Acknowledgment}
MM acknowledges funding by the BMBF (grant no. 05M22VSA) and acknowledges support by the Stuttgart Center for Simulation Science. NG~acknowledges that his research was supported by funds from the Italian MUR (Ministero dell'Universit\`{a} e della Ricerca) within the PRIN 2022 Project ``Advanced numerical methods for time dependent parametric partial differential equations with applications'' and the 2022 PRIN-PNRR grant FIN4GEO. Nicola Guglielmi is affiliated to the Italian INdAM-GNCS (Gruppo Nazionale di Calcolo Scientifico).

\bibliographystyle{plain-doi}
\bibliography{journalAbbr,LibraryUMI}

\end{document}